\documentclass[a4paper]{article}%
\usepackage{amsfonts}
\usepackage{amsmath}
\usepackage{amssymb}
\usepackage{graphicx}%
\providecommand{\U}[1]{\protect\rule{.1in}{.1in}}
\newtheorem{theorem}{Theorem}

\newtheorem{corollary}[theorem]{Corollary}

\newtheorem{example}[theorem]{Example}

\newtheorem{lemma}[theorem]{Lemma}
\newtheorem{notation}[theorem]{Notation}

\newtheorem{proposition}[theorem]{Proposition}
\newtheorem{remark}[theorem]{Remark}

\newenvironment{proof}[1][Proof]{\noindent\textit{#1.} }{}
\begin{document}

\title{Spin actions and Polygon spaces}
\author{Eunjeong Lee and Jae-Hyouk Lee}
\maketitle

\begin{abstract}
In this article, we construct correspondences between polygon spaces in
Euclidean spaces of dimension $2,3,5,9\ $and the quotient spaces of
$2$-Steifel manifolds along the normed division algebra$\ \mathbb{F}$ real
$\mathbb{R}$, complex $\mathbb{C}$, quaternions $\mathbb{H}$, octonions
$\mathbb{O}$. For the purpose, we introduce Hopf map on $\mathbb{F}^{2}\ $and
consider the spin action of $SU\left(  2,\mathbb{F}\right)  $ to spinor
$\mathbb{F}^{2}\ $and the induced $SO\ $action to the Euclidean space
$\mathbb{R\oplus F}$. The correspondences are extension of the work of
Hausmann and Knutson for polygon spaces of dimension $2,3\ $and $2$%
-Grassmannians over real and complex.

\end{abstract}

\section{Introduction}

Polygon spaces are moduli spaces of closed linkages formed by $k$ vectors in
$n$-dimensional Euclidean space, gathered up to similarity transformations and
rotation. There is a discussion on linkage by Thurston and Weeks \cite{TW} and
there has been a growing interest in polygon spaces. This interest intersects
with various fields such as Hamiltonian geometry \cite{HK}, \cite{KM}
mathematical robotics \cite{Fa}, \cite{MT}, and statistical shape theory
\cite{HR}. Also, there have been its applications to physics and chemistry as
the polygon space provides a model for the moduli spaces of constraint
Brownian motions and polymers.

In particular, Hausmann and Knutson \cite{HK} established the foundation for
comprehending the topology and geometry of polygon spaces. A noteworthy aspect
is that fascinating results emerges in the three-dimensional case. They found
out that polygon space over three dimensions closely related to complex
Grassmannians. The complex Grassmannian, renowned for its rich geometric
structures, has become a subject of widespread investigation and interest. It
helps them to figure out the K\"{a}hler structures on polygon space with fixed
side-length using the moment map for the torus action on the Grassmannian.
Additionally, they establishes a link between the bending flows and the
Gel'fand-Cetlin system on the Grassmannian. This connection is then leveraged
to compute the structures of quadrilateral, pentagon, and hexagon spaces. From
those glorious results, extensive research has been conducted on
three-dimensional polygon spaces. Subsequently, there has been further
exploration into higher-dimensional polygon spaces. For instance, Foth and
Lozano \cite{FL} studied five-dimensional polygon spaces, establishing a
connection with the GIT quotient of quaternionic projective lines through the
diagonal action of $SL(2,\mathbb{H})$.

In this paper, we extend the construction of polygon spaces for $2$- and
$3$-dimensional euclidean spaces as real$\mathbb{\ }$and complex Grassmannians
in \cite{HK} to the polygon spaces in $5$- and $9$-dimensional Euclidean
spaces along the quaternions $\mathbb{H\ }$and octonions $\mathbb{O}$.

In \cite{HK}, Hausmann and Knutson consider $2$- and $3$-dimensional euclidean
spaces as $\mathbb{C}$ and $\operatorname{Im}\mathbb{H}$. In particular,
$\operatorname{Im}\mathbb{H\ }$appears along the Hopf map $\mathbb{H}%
\rightarrow\operatorname{Im}\mathbb{H}$ so that the $U\left(  2\right)
\ $action to $\mathbb{H}$ (as $\mathbb{C}^{2}$) induces $SO\left(  3\right)
\ $action to $\operatorname{Im}\mathbb{H~}$(as $\mathbb{R}^{3}$). Therefrom,
it is obtained the correspondence of the complex $2$-grassmannian and the
polygon space over $\mathbb{R}^{3}$. Moreover, the octonionic version of this
is proposed in \cite{HK}. Here, we observe that $U\left(  2\right)  \ $action
to $\mathbb{C}^{2}\ $is rather $Spin\left(  3\right)  \simeq SU\left(
2\right)  \ $action to spinor $\mathbb{C}^{2}$. Moreover this special
isomorphism of spin group is a case of the followings%

\[%
\begin{tabular}
[c]{|c||c|c|c|c|}\hline
$\mathbb{F}$ & $\mathbb{R}$ & $\mathbb{C}$ & $\mathbb{H}$ & $\mathbb{O}%
$\\\hline
$SU(2,\mathbb{F})$ & $SO(2)\simeq Spin\left(  2\right)  $ & $SU(2)\simeq
Spin\left(  3\right)  $ & $Sp(2)\simeq Spin\left(  5\right)  $ & $Spin\left(
9\right)  $\\\hline
Spinors & $\mathbb{R}^{2}$ & $\mathbb{C}^{2}$ & $\mathbb{H}^{2}$ &
$\mathbb{O}^{2}$\\\hline
\end{tabular}
\]
where $Spin\left(  9\right)  \ $is denoted by $SU(2,\mathbb{O})\mathbb{\ }%
$along this table. From our approach, the results in \cite{HK} are better
understood as $\mathbb{R}$ and $\mathbb{C\ }$rather than $\mathbb{C}$ and
$\mathbb{H}$. Thus it is natural that the extensions are considered for
$\mathbb{H}$ and $\mathbb{O}$ in the sense of normed division algebra.

First of all, we introduce new Hopf maps $\Phi\ $and identification $\pi\ $for
$\mathbb{F}=\mathbb{R},\mathbb{C},\mathbb{H}$
\[%
\begin{array}
[c]{cccccc}%
\pi\circ\Phi: & \mathbb{F}^{2} & \rightarrow & \mathcal{H}_{2}^{0}%
(\mathbb{F}) & \rightarrow & \mathbb{R}\oplus\mathbb{F}\\
&
\begin{pmatrix}
x\\
y
\end{pmatrix}
& \mapsto &
\begin{pmatrix}
\frac{|x|^{2}-|y|^{2}}{2} & x\overline{y}\\
y\overline{x} & \frac{|y|^{2}-|x|^{2}}{2}%
\end{pmatrix}
& \mapsto & \left(  \frac{|x|^{2}-|y|^{2}}{2},x\overline{y}\right)
\end{array}
\]
where $\mathcal{H}_{2}^{0}(\mathbb{F})$ is the space of $2\times2$ Hermitian
traceless matrices over $\mathbb{F}$. By Lemma \ref{surjection}, the Hopf maps
$\Phi\ $are surjective, and the generic fibers of $\Phi$ are
\[
\mathbb{F}(1):=\{c\in\mathbb{F}:|c|=1\}
\]
which are compatible to the (right)$\ \mathbb{F}$-action to $\mathbb{F}^{2}$.
We also introduce Hopf map $\Phi_{\mathbb{O}}$ and obtain the similar results
(Lemma\ref{OctO1SuProp}, Lemma \ref{OctonionFiberLemma}). Remark that since
$\mathbb{O}$ is not associative, $\mathbb{O}(1)\ $is not a group but a Moufang
loop so that we need to modify the right $\mathbb{O}$-multiplication to
$\mathbb{O}^{2}$ to our purpose.

For $\mathbb{F}=\mathbb{R},\mathbb{C},\mathbb{H}$, the spin action of
$SU(2,\mathbb{F})\ $to $\mathbb{F}^{2}$ induces $SO\left(  1+\dim
\mathbb{F}\right)  $ action on $\mathbb{R}\oplus\mathbb{F\ }$via $\pi\circ
\Phi\ $by Proposition \ref{SORCH}. Similarly, we obtain the spin action of
$SU(2,\mathbb{O})\simeq Spin\left(  9\right)  \ $inducing $SO\left(
1+\dim\mathbb{O}\right)  $ action on $\mathbb{R}\oplus\mathbb{O}$. Here again
since $\mathbb{O}$ is not associative, we need another argument applying the
generators $\mathbb{R}\oplus\mathbb{O}\left(  1\right)  \ $of $Spin\left(
9\right)  \ $(subsection \ref{subsectionOandGroupaction}).

Finally, we apply these for $2$-Stiefel manifold $V_{\mathbb{F}}(2,k)$ as the
set of $2\times k$ matrices $X$ over $\mathbb{F}$ such that $XX^{\ast}=I_{2}$.
This leads to the commutative diagram for $2$-Stiefel manifold $V_{\mathbb{F}%
}(2,k)$ and polygon spaces $\widetilde{\mathcal{P}}_{k}(\mathbb{R}%
^{1+dim\mathbb{F}})$
\[%
\begin{matrix}
V_{\mathbb{F}}(2,k)/\mathbb{F}(1)^{k} & \xrightarrow{ \widetilde{\Phi}^k} &
\mathcal{M}_{k}(\mathbb{R}^{1+dim\mathbb{F}})\\
&  & \\
SU(2,\mathbb{F})\ \ \big\downarrow &  & \big\downarrow\ \ SO(1+dim\mathbb{F}%
)\\
&  & \\
SU(2,\mathbb{F})\backslash V_{\mathbb{F}}(2,k)/\mathbb{F}(1)^{k} &
\xrightarrow{\simeq} & \widetilde{\mathcal{P}}_{k}(\mathbb{R}^{1+dim\mathbb{F}%
})
\end{matrix}
.
\]
Therefrom, we obtain the following identifications%
\begin{align*}
\widetilde{\mathcal{P}}_{k}(\mathbb{R}^{2})  &  \simeq\widetilde
{Gr}_{\mathbb{R}}(2,k)/O(1)^{k},\\
\widetilde{\mathcal{P}}_{k}(\mathbb{R}^{3})  &  \simeq SU(2)\backslash
V_{\mathbb{C}}(2,k)/U(1)^{k},\\
\widetilde{\mathcal{P}}_{k}(\mathbb{R}^{5})  &  \simeq Gr_{\mathbb{H}%
}(2,k)/\mathbb{H}(1)^{k},\\
\widetilde{\mathcal{P}}_{k}(\mathbb{R}^{9})  &  \simeq SU(2,\mathbb{O}%
)\backslash V_{\mathbb{O}}(2,k)/\mathbb{O}(1)^{k}.
\end{align*}
Remark that the fundamental properties of $\mathbb{O\ }$are also presented in
detail(subsection \ref{SubsectionOct}) as this article involves various
calculations of $\mathbb{O\ }$relied on them.

\bigskip

\section{Preliminaries}

\subsection{Polygon spaces}

Polygons in this article are closed figures in a finite-dimensional Euclidean
space $(V,\langle,\rangle)$, formed by connecting more than equal to three
vertices with directed straight line segments which are edges of the polygon.
In this article we consider polygons in a finite-dimensional Euclidean space
$(V,\langle,\rangle)\ $as a map whose values are vectors in $V\ $presenting
directed line segments of the polygons. A $k$-gon $p$ is a nonzero map%
\[%
\begin{array}
[c]{cccc}%
p: & \{1,\cdots,k\} & \rightarrow & V\\
& i & \mapsto & p(i)
\end{array}
\]
satisfying closed condition
\[
\sum_{i=1}^{k}p(i)=0.
\]
The definition allows polygons with self-intersection, knotted polygons. The
value of the map $p\ $can be the zero vector in $V$ but this may not occur
simultaneously, and the polygons are possibly degenerate as in \cite{KM2}.

On the set $M_{k}\left(  V\right)  \ $of all $k$-gons in $V$, there are
natural actions of the orthogonal group $O(V)\ $(also $SO(V)$) and the
multiplicative group $(\mathbb{R}^{+},\times)\ $defined respectively by%
\begin{align*}
\left(  g\cdot p\right)  (i)  &  :=g\left(  p(i)\right)  \text{,\ for }g\in
O(V),i\in\{1,\cdots,k\}\\
\left(  c\cdot p\right)  (i)  &  :=cp(i)\text{,\ for }\mathbb{R}^{+}%
,i\in\{1,\cdots,k\}.
\end{align*}

With a choose of an ordered basis of\ a $n$-dimensional vector space $V\ $over
$\mathbb{R}$, polygons $p\ $can be presented as $n\times k\ $matrices\
\[
\left[  p\right]  :=%
\begin{pmatrix}
| &  & |\\
p(1) & \cdots & p\left(  k\right) \\
| &  & |
\end{pmatrix}
\]
where $p\left(  i\right)  $'s are considered as column vectors presented by
the basis of $V$. Moreover the set of $k$-gons $M_{k}\left(  V\right)  \ $is
regarded as the subset of $n\times k\ $matrices $\operatorname{Mat}_{n\times
k}(\mathbb{R})\ $over real numbers.

Let $\mathcal{M}_{k}(V)$ is the quotient space of the set of all $k$-gons in
$M_{k}\left(  V\right)  $ by the $(\mathbb{R}^{+},\times)\ $action, and we
identify $\mathcal{M}_{k}(V)$ as the space of closed polygons in $V$ with
fixed perimeter%
\[
\sum_{i=1}^{k}\left\vert p(i)\right\vert =1
\]
where $\left\vert \ \right\vert \ $is the Euclidean norm of $V$.

The \textbf{polygon spaces} of $k$-gons\ in the Euclidean space $(V,\langle
,\rangle)\ $are defined as the quotient space of $\mathcal{M}_{k}(V)$ by
$SO(V),O(V)$:
\[
\mathcal{P}_{k}(V):=O(V)\backslash\mathcal{M}_{k}(V),\;\mathcal{\tilde{P}}%
_{k}(V):=SO(V)\backslash\mathcal{M}_{k}(V)
\]

\bigskip

\subsection{Stiefel manifold}

For $\mathbb{F}=\mathbb{R},\mathbb{C},\mathbb{H}$, the set
\[
V_{\mathbb{F}}(k,n):=\{X\in\operatorname{Mat}_{k\times n}(\mathbb{F})\mid
XX^{\ast}=I_{k}\}
\]
of ordered orthonormal $k$-frames in $\mathbb{F}^{n}$\ is called the
\textbf{Stiefel manifold}. Here, $X^{\ast}$ denotes the conjugate transpose of
$X$\ .We define conjugation on $\mathbb{F}$ by reversing the sign of the
imaginary part while preserving the real part. For example, for
$q=a+bi+cj+dk\in\mathbb{H}$, $a,b,c,d\in\mathbb{R}$, its conjugation
$\overline{q}$ is%
\[
\overline{q}:=a-bi-cj-dk.
\]
Moreover, we use the standard inner product on $\mathbb{F}^{n}$given by%
\[
\langle x,y\rangle=\sum_{i=1}^{n}x_{i}\overline{y}_{i}%
\]
$\ $for $x=\left(  x_{1},x_{2},...,x_{n}\right)  ,y=\left(  y_{1}%
,y_{2},...,y_{n}\right)  \ $in $\mathbb{F}^{n}$. This inner product is
positive definite with signature $(n,0)\ $and the condition $XX^{\ast}=I_{k}$
ensures that the rows of matrices $X$ in $V_{\mathbb{F}}(k,n)\ $form an
orthonormal set of $k\ $vectors of $\mathbb{F}^{n}$.

\begin{remark}
Since $\mathbb{F}=\mathbb{H\ }$is noncommutative, it is important to fix
$\mathbb{H}^{n}\ $as left or right $\mathbb{H}$-module. From above definition
of norm and $V_{\mathbb{H}}(k,n)$, the natural choice is right $\mathbb{H}%
$-module. It is also matched with matrix multiplication from the left to
column vectors of $\mathbb{H}^{n}$.
\end{remark}

The Stiefel manifold $V_{\mathbb{F}}(k,n)$ is also a homogenous space. Since
$V_{\mathbb{F}}(k,n)$ is the set of all $k$-frames of $\mathbb{F}^{n}$, the
Lie groups of orthogonal matrices for each $\mathbb{F}$
\[
O(n),U(n),Sp(n)
\]
act transitively on $V_{\mathbb{F}}(k,n)$ for $\mathbb{F}=\mathbb{R}%
,\mathbb{C},\mathbb{H}$, respectively. Then the isotropy subgroup of a given
$k$-frame is
\[
O(n-k),U(n-k),Sp(n-k).
\]
Thus, the Stiefel manifolds can be expressed as the homogeneous spaces%
\begin{align*}
V_{\mathbb{R}}(k,n)  &  \simeq O(n)/O(n-k)\\
V_{\mathbb{C}}(k,n)  &  \simeq U(n)/U(n-k)\\
V_{\mathbb{H}}(k,n)  &  \simeq Sp(n)/Sp(n-k).
\end{align*}

Now, the Stiefel manifolds $V_{\mathbb{F}}(k,n)\ $are naturally corresponded
to the \textbf{Grassmannian manifold} $Gr_{\mathbb{F}}(k,n)$, the set of all
$k$-dimensional subspaces of $\mathbb{F}^{n}$. This space also admits a
transitive action of
\[
O(n),U(n),Sp(n),
\]
but its isotropy subgroup is
\[
O(n-k)\times O(k),\quad U(n-k)\times U(k),\quad Sp(n-k)\times Sp(k).
\]
Hence, the Grassmannian can be described as the homogeneous space and can be
obtained as a quotient of the Stiefel manifold.%
\begin{align*}
Gr_{\mathbb{R}}(k,n)  &  \simeq O(n)/\left(  O(n-k)\times O(k)\right)  \simeq
V_{\mathbb{R}}(k,n)/O(k)\\
Gr_{\mathbb{C}}(k,n)  &  \simeq U(n)/\left(  U(n-k)\times U(k)\right)  \simeq
V_{\mathbb{C}}(k,n)/U(k)\\
Gr_{\mathbb{H}}(k,n)  &  \simeq Sp(n)/\left(  Sp(n-k)\times Sp(k)\right)
\simeq V_{\mathbb{H}}(k,n)/Sp(k).
\end{align*}

In addition, for $\mathbb{F}=\mathbb{R}$, a point in the $k$-Stiefel manifold
$V_{\mathbb{R}}(k,n)\ $represents an ordered orthonormal $k$-frame. By fixing
an orientation of $\mathbb{R}^{n}$, $V_{\mathbb{R}}(k,n)\ (k<n)\ $can also be
described as
\[
V_{\mathbb{R}}(k,n)\simeq SO(n)/SO(n-k)
\]
which is the $SO(k)\ $principal bundle over the \textit{oriented Grassmannian}
$\widetilde{Gr}_{\mathbb{R}}(k,n)$
\[
\widetilde{Gr}_{\mathbb{R}}(k,n)\simeq SO(n)/\left(  SO(n-k)\times
SO(k)\right)  \simeq V_{\mathbb{R}}(k,n)/SO(k)
\]

\subsection{\label{SubsectionOct}Octonions}

In this article octonions $\mathbb{O}$ is a real algebra consisting of unit
$1$ and seven independent vectors $e_{i}\ \left(  i=1,2,...,7\right)  $
generating the algebra with relations
\begin{align*}
e_{i}^{2}  &  =-1,e_{i}e_{i+1}=e_{i+3},e_{i+1}e_{i}=-e_{i+3}\\
e_{i+1}e_{i+3}  &  =e_{i},\ e_{i+3}e_{i+1}=-e_{i},\ e_{i+3}e_{i}%
=e_{i+1},\ e_{i}e_{i+3}=-e_{i+1}%
\end{align*}
for $i=1,2,...,7$.

Therefrom it is easy to see that the algebra of $\mathbb{O\ }$is neither
commutative nor associative. This description of $\mathbb{O\ }$is intuitive
and useful to catch the symmetry of $\mathbb{O}$, but we use another
description of $\mathbb{O\ }$as a composition algebra on behalf of the base
free argument.

\bigskip

\begin{center}
\textbf{Normed algebras and Octonions}
\end{center}

Here we consider octonions $\mathbb{O\ }$as a normed algebra which is a case
of composition algebras. We have a brief introduction to normed algebras and
basic properties. We consider finite dimensional Euclidean vector space
$V\ $with positive definite inner product $\left\langle ,\right\rangle $. The
norm $\left\Vert x\right\Vert \ $of vector $x\ $in $V\ $is defined
\[
\left\Vert x\right\Vert :=\sqrt{\left\langle x,x\right\rangle },
\]
and the square norm $\left\Vert x\right\Vert ^{2}\ $is a \textit{quadratic
form\ }of $V$ which is a map satisfying $\left\Vert \lambda x\right\Vert
^{2}=\lambda^{2}\left\Vert x\right\Vert ^{2}\ $for $\lambda\in F$,\ $x\in
V$\ and $\left\Vert x+y\right\Vert ^{2}-\left\Vert x\right\Vert ^{2}%
-\left\Vert y\right\Vert ^{2}$ $\left(  x,y\in V\right)  \ $is bilinear. In
fact, by polarization
\[
\left\langle x,y\right\rangle =\frac{1}{2}\left\{  \left\Vert x+y\right\Vert
^{2}-\left\Vert x\right\Vert ^{2}-\left\Vert y\right\Vert ^{2}\right\}  .
\]

We consider algebras$\ $defined over a finite dimensional vector space $V$
over $\mathbb{R}$ with a multiplication $\cdot$ and its unit element $e$. The
algebra $A$ is called a \textit{normed algebra} if the multiplication\ is
compatible with the norm $\left\Vert {}\right\Vert \ $as
\[
\left\Vert a\cdot b\right\Vert =\left\Vert a\right\Vert \left\Vert
b\right\Vert \;,\;a,b\in A.
\]
This implies $\left\Vert e\right\Vert =1$ and $A$ is a division algebra,
namely for $a$, $b$ $\in$ $A$ with $a\cdot b=$ $0$ , then either $a$ $=0$ or
$b=0$. By this condition, we can perform the left or right cancellation.

\begin{notation}
The unit element $e\ $is also denoted by $1\ $unless there is confusion. It is
important to observe in Euclidean inner product $\left\langle ,\right\rangle
\ $is positive definite. One may begin with a quadratic form and induced
bilinear form. Here the bilinear form is supposed to be nondegenerated. We
also denote $a\cdot b\ $by $ab\ $unless there is confusion.
\end{notation}

\begin{example}
For the octonions $\mathbb{O\ }$defined above, if we consider a norm
$\left\Vert \ \ \right\Vert \ $induced from the inner product$\ $on
$\mathbb{O}$ such that $\ \left\{  e,e_{i}\ \left(  i=1,2,...,7\right)
\right\}  \ $forms orthonormal basis, namely,\
\[
\left\Vert a_{0}e+\sum_{i=1}^{7}a_{i}e_{i}\right\Vert =\sqrt{a_{0}^{2}%
+\sum_{i=1}^{7}a_{i}^{2}},
\]
we can see $\mathbb{O}$ is a normed algebra by the direct computation.
Similarly we conclude that $\mathbb{R},\mathbb{C},\mathbb{H\ }$are normed
algebra for the standard norm.
\end{example}

In below, we list basic properties of normed algebras which are especially
useful to handle octonions. The proofs of properties will be presented roughly
for leader to grab the general idea of octonions which appear ubiquitous in
this article. Reader can find the further detail in \cite{Spr} and \cite{Har}

\begin{lemma}
\label{OCTBasicLemma}Let $A\ $be a normed algebra with the unit $e$. Then for
$a,b,$ $x,y\in A$

(1) $\left\langle ax,bx\right\rangle =\left\langle a,b\right\rangle \left\Vert
x\right\Vert ^{2},\;\left\langle ya,yb\right\rangle =\left\Vert y\right\Vert
^{2}\left\langle a,b\right\rangle $

(2) $2\left\langle a,b\right\rangle \left\langle x,y\right\rangle
=\left\langle ax,by\right\rangle +\left\langle ay,bx\right\rangle $

(3) $xy+yx=2\left\langle y,e\right\rangle x+2\left\langle x,e\right\rangle
y-2\left\langle x,y\right\rangle e$

(4) (\textit{rank equation}) $x^{2}=2\left\langle x,e\right\rangle
x-\left\Vert x\right\Vert ^{2}e$ $\;$
\end{lemma}

\begin{proof}
(1) By the condition of normed algebras and its linearization,
\begin{align*}
\left\Vert \left(  a+b\right)  x\right\Vert ^{2}  &  =\left\Vert a\right\Vert
^{2}\left\Vert x\right\Vert ^{2}+2\left\langle ax,bx\right\rangle +\left\Vert
b\right\Vert ^{2}\left\Vert x\right\Vert ^{2},\\
\left\Vert \left(  a+b\right)  x\right\Vert ^{2}  &  =\left\Vert
a+b\right\Vert ^{2}\left\Vert x\right\Vert ^{2}=\left(  \left\Vert
a\right\Vert ^{2}+2\left\langle a,b\right\rangle +\left\Vert b\right\Vert
^{2}\right)  \left\Vert x\right\Vert ^{2},
\end{align*}
hence
\[
\left\langle ax,bx\right\rangle =\left\langle a,b\right\rangle \left\Vert
x\right\Vert ^{2}.
\]
In a similar way
\[
\left\langle ya,yb\right\rangle =\left\Vert y\right\Vert ^{2}\left\langle
a,b\right\rangle .
\]

(2) Replace $x$ by $x+y$ in (1).

(3) By applying (2)%
\begin{align*}
\left\langle xy,ae\right\rangle +\left\langle yx,ae\right\rangle  &
=2\left\langle x,a\right\rangle \left\langle y,e\right\rangle +2\left\langle
y,a\right\rangle \left\langle x,e\right\rangle -\left\langle x,ay\right\rangle
-\left\langle y,ax\right\rangle \\
&  =\left\langle 2\left\langle y,e\right\rangle x+2\left\langle
x,e\right\rangle y,a\right\rangle -\left\{  \left\langle ex,ay\right\rangle
+\left\langle ey,ax\right\rangle \right\} \\
&  =\left\langle 2\left\langle y,e\right\rangle x+2\left\langle
x,e\right\rangle y,a\right\rangle -2\left\langle x,y\right\rangle \left\langle
e,a\right\rangle \\
&  =\left\langle 2\left\langle y,e\right\rangle x+2\left\langle
x,e\right\rangle y-2\left\langle x,y\right\rangle ,a\right\rangle .
\end{align*}
and since above is true for any $a\in A,$ we get (3).

(4) Set $y=x$. \ \ \ \ \ \ \ \ \ \ \ \ \ \ \ \ \ \ \ \ \ \ \ \ $\blacksquare$
\end{proof}

Note above Lemma \ref{OCTBasicLemma} implies that $xy=-yx$ if $x\perp y$ and
normed algebras are \textit{power associative} which means by any subalgebra
generated by any one elements is associative.

In a normed algebra $A$, we can define \textbf{conjugation} for each $a$ $\in$
$A$ as follow,
\[
\bar{a}:=2\left\langle a,e\right\rangle e-a.
\]
and we call $\left\langle a,e\right\rangle e$ as the real part of $a$ (denote
$\operatorname{Re}a$ )and $a-\left\langle a,e\right\rangle e$ as imaginary
part of $a$ (denote $\operatorname{Im}a$). Note conjugation preserves any subalgebra.

\begin{lemma}
\label{conjugationLemma}For $x$, $y,z$ of normed algebra $A,$
\begin{align*}
x\bar{x}  &  =\bar{x}x=\left\Vert x\right\Vert ^{2}e=\left\Vert \bar
{x}\right\Vert ^{2}e,\ \overline{xy}=\bar{y}\bar{x}\;,\;\overline{\bar{x}}=x\\
\overline{x+y}  &  =\bar{x}+\bar{y},\;\left\langle \bar{x},\bar{y}%
\right\rangle =\left\langle x,y\right\rangle ,\ \left\langle x,y\right\rangle
e=\operatorname{Re}\left(  x\bar{y}\right)  e
\end{align*}
and
\[
\left\langle xy,z\right\rangle =\left\langle x,z\bar{y}\right\rangle
=\left\langle y,\bar{x}z\right\rangle
\]

\end{lemma}

\begin{proof}
\ One can get the above by straight forward calculation with the definition of
conjugation. We show the last one.
\begin{align*}
\left\langle y,\bar{x}z\right\rangle  &  =\left\langle y,\left(  2\left\langle
x,e\right\rangle e-x\right)  z\right\rangle =2\left\langle x,e\right\rangle
\left\langle y,z\right\rangle -\left\langle y,xz\right\rangle \\
&  =\left\langle xy,ez\right\rangle +\left\langle xz,ey\right\rangle
-\left\langle y,xz\right\rangle =\left\langle xy,z\right\rangle .
\end{align*}

\ $\blacksquare$
\end{proof}

From above lemma we derive the following Lemma,

\begin{lemma}
\label{ConjugationLemma2}\ For all $x,$ $y,$ $z$ of a normed algebra $A,$%
\begin{align*}
x\left(  \bar{x}y\right)   &  =\left\Vert x\right\Vert ^{2}y\;,\;\left(
x\bar{y}\right)  y=x\left\Vert y\right\Vert ^{2}\\
x\left(  \bar{y}z\right)  +y\left(  \bar{x}z\right)   &  =2\left\langle
x,y\right\rangle z,\;\left(  x\bar{y}\right)  z+\left(  x\bar{z}\right)
y=2\left\langle y,z\right\rangle x
\end{align*}

\end{lemma}

\begin{proof}
The first identity holds since for any $z$ $\in$ $A$ ,
\[
\left\langle x\left(  \bar{x}y\right)  ,z\right\rangle =\left\langle \bar
{x}y,\bar{x}z\right\rangle =\left\Vert \bar{x}\right\Vert ^{2}\left\langle
y,z\right\rangle =\left\langle \left\Vert x\right\Vert ^{2}y,z\right\rangle .
\]
And the remains are conjugation and linearization of the first
two.\ \ \ $\blacksquare$
\end{proof}

From this Lemma \ref{ConjugationLemma2}, we can easily obtain the following.

\begin{corollary}
\label{AlternativeCoro}(Alternative laws) For all $x,$ $y$ of a normed algebra
$A,$%
\[
\left(  xy\right)  x=x\left(  yx\right)  ,\ x\left(  xy\right)  =x^{2}%
y,\left(  xy\right)  y=xy^{2}%
\]

\end{corollary}

Based on above lemmas, we can get produce another proof of Artin's theorem.

\begin{proposition}
\label{Artin}(E. Artin) The subalgebra generated by any two elements of a
normed algebra is associative, namely it is \textit{alternative}.
\end{proposition}

\begin{proof}
Since normed algebras are power associative, we may assume the subalgebra $B$
generated by $a$, $b$ such that $a\perp b,$ $a\perp e$ and $b\perp e$. Then
$B$ contains another vector $ab$ which is orthogonal to the vector space span
by $e$, $a\;$and $b$ because
\[
\left\langle ab,e\right\rangle =\left\langle a,\bar{b}\right\rangle
=-\left\langle a,b\right\rangle =0,\left\langle ab,a\right\rangle =\left\Vert
a\right\Vert ^{2}\left\langle b,e\right\rangle =0.
\]
Note that for any $x$ such that $x\perp e$,
\[
x^{2}=-x\left(  -x\right)  =\left(  -x\right)  \bar{x}=-\left\Vert
x\right\Vert ^{2}e.
\]
From this and lemma \ref{ConjugationLemma2}, we see $B$ doesn't have more
independent vectors. Now for the associativity, we only need to see the
associativity among $a$, $b$ and $ab$. By the alternative laws in Corollary
\ref{AlternativeCoro}%
\[
a\left(  ab\right)  =a^{2}b,a\left(  ba\right)  =\left(  ab\right)  a,
\]
and $ab=\left(  -b\right)  a\ $since $a\perp b$. Therefore, we only need to
check a case $a\cdot\left(  b\cdot\left(  a\cdot b\right)  \right)  $.
\begin{align*}
a\left(  b\left(  ab\right)  \right)   &  =a\left(  \left(  -b\right)  \left(
ba\right)  \right)  =a\left(  b\left(  \bar{b}a\right)  \right)  =\left\Vert
b\right\Vert ^{2}a^{2}\\
&  =-\left\Vert a\right\Vert ^{2}\left\Vert b\right\Vert ^{2}=-\left\Vert
ab\right\Vert ^{2}=-\left(  ab\right)  \overline{\left(  ab\right)  }=\left(
\left(  ab\right)  \left(  ab\right)  \right)  .
\end{align*}

\ $\blacksquare$
\end{proof}

The following identities play important roles for studying normed algebra.

\begin{proposition}
\label{MoufangPro}(Moufang identities) For any $a,x,y$ of a normed algebra
$A,$ the following identities hold
\[
\left(  ax\right)  \left(  ya\right)  =a\left(  xy\right)  a,\ a\left(
x\left(  ay\right)  \right)  =\left(  axa\right)  y,\ x\left(  aya\right)
=\left(  \left(  xa\right)  y\right)  a
\]

\end{proposition}

\begin{proof}
Without long generality, it is enough to consider
\[
\left(  ax\right)  \left(  ya\right)  =a\left(  \left(  xy\right)  a\right)
,\ a\left(  x\left(  ay\right)  \right)  =\left(  a\left(  xa\right)  \right)
y
\]
For arbitrary $z\in A$,%
\begin{align*}
\left\langle \left(  ax\right)  \left(  ya\right)  ,z\right\rangle  &
=\left\langle ya,\overline{\left(  ax\right)  }z\right\rangle =2\left\langle
y,\overline{ax}\right\rangle \left\langle a,z\right\rangle -\left\langle
yz,\overline{\left(  ax\right)  }a\right\rangle \\
&  =2\left\langle xy,\bar{a}\right\rangle \left\langle a,z\right\rangle
-\left\Vert a\right\Vert ^{2}\left\langle yz,\bar{x}\right\rangle
\end{align*}
and
\begin{align*}
\left\langle a\left(  \left(  xy\right)  a\right)  ,z\right\rangle  &
=\left\langle \left(  xy\right)  a,\bar{a}z\right\rangle =2\left\langle
xy,\bar{a}\right\rangle \left\langle a,z\right\rangle -\left\langle \left(
xy\right)  z,\bar{a}a\right\rangle \\
&  =2\left\langle xy,\bar{a}\right\rangle \left\langle a,z\right\rangle
-\left\Vert a\right\Vert ^{2}\left\langle xy,\bar{z}\right\rangle .
\end{align*}
Since $\left\langle yz,\bar{x}\right\rangle =\left\langle xy,\bar
{z}\right\rangle $, we have $\left(  ax\right)  \left(  ya\right)  =a\left(
xy\right)  a$.

Again for arbitrary $z\in A$,
\begin{align*}
\left\langle \left(  a\left(  xa\right)  \right)  y,z\right\rangle  &
=\left\langle xa,\bar{a}\left(  z\bar{y}\right)  \right\rangle =\left\langle
x,\left(  \bar{a}\left(  z\bar{y}\right)  \right)  \bar{a}\right\rangle
=\left\langle x,\left(  \bar{a}z\right)  \left(  \bar{y}\bar{a}\right)
\right\rangle \\
&  =\left\langle x\left(  ay\right)  ,\left(  \bar{a}z\right)  \right\rangle
=\left\langle a\left(  x\left(  ay\right)  \right)  ,z\right\rangle .
\end{align*}
Thus we have $a\left(  x\left(  ay\right)  \right)  =\left(  axa\right)
y$.\ \ \ $\blacksquare$
\end{proof}

\bigskip

\begin{remark}
\label{MoufangRmk}In general $\left(  xa\right)  \left(  ay\right)  =x\left(
a^{2}y\right)  ,$ $\left(  xa\right)  \left(  ay\right)  =\left(
xa^{2}\right)  y$ $\ $may not be true. For octonions,%
\begin{align*}
\left(  e_{1}e_{2}\right)  \left(  e_{2}e_{3}\right)   &  =e_{4}e_{5}=e_{7}\\
e_{1}\left(  e_{2}^{2}e_{3}\right)   &  =-e_{1}e_{3}=-e_{7}=\left(  e_{1}%
e_{2}^{2}\right)  e_{3}%
\end{align*}

\end{remark}

We also state a theorem of classification of normed algebras.

\begin{theorem}
(Hurwitz) The only normed algebras over $\mathbb{R}$ with its norm inducing
positive definite inner product are real, complex, quaternions and octonions.
\end{theorem}

Even though octonions $\mathbb{O\ }$is member of the classification, the
natural extensions of Grassmannians and Steifel manifolds from quaternions are
not available since $\mathbb{O\ }$is nonassociative. In spite of this barrier,
one can still consider projective lines and planes corresponding to
$\mathbb{O}^{2}\ $and $\mathbb{O}^{3}\ $along Hopf fibrations. In this
article, we apply these to study certain subsets of matrices defined over
$\mathbb{O\ }$to construct polygon spaces.
\[
V_{\mathbb{O}}(k,n):=\{X\in\operatorname{Mat}_{k\times n}(\mathbb{O})\mid
XX^{\ast}=I_{k}\}.
\]
Here $V_{\mathbb{O}}(k,n)\ $is considered as a manifold induced from the
imbedding of $\operatorname{Mat}_{k\times n}(\mathbb{O})\ $to real vector space.

\begin{remark}
Since $\mathbb{F}=\mathbb{O\ }$is noncommutative and nonassociative,
$\mathbb{O}^{n}\ $is not a module in general. Even though, we consider
$\mathbb{O}$-multiplication to $\mathbb{O}^{n}$ and the choice is fixed to the
right as $\mathbb{H}^{n}\ $as $\mathbb{H}$-module.
\end{remark}

\section{Hopf Maps}

For $\mathbb{F}=\mathbb{R},\mathbb{C},\mathbb{H}$ and positive integer $n$,
and for $\mathbb{F}=\mathbb{O}$ only when $n=1,2,3$, we consider a map
\begin{align*}
\phi:\mathbb{F}^{n}  &  \rightarrow\mathcal{H}_{n}(\mathbb{F})\\
v  &  \mapsto vv^{\ast}%
\end{align*}
where $\mathcal{H}_{n}(\mathbb{F})$ is space of hermitian $n\times n$ matrices
over $\mathbb{F}$ and $v$ is written as a column. Observe that $\phi
(v)=vv^{\ast}$ is a rank $1$ projection for $v\in S^{dn-1}\subset
\mathbb{F}^{n}$ where $d=\dim_{\mathbb{R}}\mathbb{F}$. This gives the Hopf
fibration:%
\[%
\begin{array}
[c]{cccc}%
\{c\in\mathbb{F}:|c|=1\} & \rightarrow & \{v\in\mathbb{F}^{n}:|v|=1\} &
\subset\mathbb{F}^{n}\\
&  &  & \\
&  & \downarrow\phi & \\
&  &  & \\
&  & \mathbb{F}P^{n-1} &
\end{array}
\]
However, for $\mathbb{F}=\mathbb{O}$, its non-associativity prevents it from
defining an equivalence relation in the same way as $\mathbb{R},\mathbb{C}$
and $\mathbb{H}$. It is well-known that the octonionic projective space
$\mathbb{O}P^{n-1}$ does not exist for $n>3$.

For $\mathbb{F}=\mathbb{R},\mathbb{C},\mathbb{H}$, $\mathbb{O}$, the fibers
are unit spheres and denoted as
\[
\mathbb{F}(1):=\{c\in\mathbb{F}:|c|=1\}=S^{d-1}.
\]
Here the fibers $\mathbb{F}(1)$ have a group structure for $\mathbb{F}%
=\mathbb{R},\mathbb{C},\mathbb{H\ }$but $\mathbb{O}(1)\ $is not a group but a
\textit{Moufang loop} even though $\mathbb{O}(1)\ $is closed under the
multiplication since $\mathbb{O\ }$is not associative.
\[%
\begin{tabular}
[c]{|c||c|c|c|c|}\hline
$\mathbb{F}$ & $\mathbb{R}$ & $\mathbb{C}$ & $\mathbb{H}$ & $\mathbb{O}%
$\\\hline
$\mathbb{F}(1)$ & $%
\begin{array}
[c]{c}%
S^{0}\\
O(1)
\end{array}
$ & $%
\begin{array}
[c]{c}%
S^{1}\\
U(1)\simeq SO\left(  2\right)
\end{array}
$ & $%
\begin{array}
[c]{c}%
S^{3}\\
Sp(1)\simeq SU\left(  2\right)
\end{array}
$ & $S^{7}$\\\hline
\end{tabular}
\]

\subsection{Hopf maps and $\mathbb{R},\mathbb{C},\mathbb{H}$}

The case $n=2$ corresponds to the classical Hopf fibration, which is special
in many ways and has been extensively studied. While we do not discuss it in
detail here, we slightly modify the projection map $\phi$ to fit our purposes.
In \cite{HK}, the Hopf map on quaternion $\mathbb{H}\rightarrow
\operatorname{Im}(\mathbb{H})$ was introduced as part of their study on
polygon spaces in three dimensions. Now, we adjust $\phi$ for the case $n=2$
and $\mathbb{F}=\mathbb{R},\mathbb{C},\mathbb{H}$, and define our version of
Hopf map $\Phi\ $for $\mathbb{F}=\mathbb{R},\mathbb{C},\mathbb{H}$:%
\[%
\begin{array}
[c]{cccc}%
\Phi: & \mathbb{F}^{2} & \rightarrow & \mathcal{H}_{2}^{0}(\mathbb{F})\\
&  &  & \\
& v & \mapsto & vv^{\ast}-\frac{tr(vv^{\ast})}{2}I_{2}%
\end{array}
\]
where $\mathcal{H}_{2}^{0}(\mathbb{F})$ denotes the space of $2\times2$
hermitian traceless matrices over $\mathbb{F\ }$and $v\ $is considered as a
column vector. Explicitly, the image is given by:
\[
\Phi\left(
\begin{pmatrix}
x\\
y
\end{pmatrix}
\right)  =%
\begin{pmatrix}
\frac{|x|^{2}-|y|^{2}}{2} & x\overline{y}\\
y\overline{x} & \frac{|y|^{2}-|x|^{2}}{2}%
\end{pmatrix}
\]
We define an isomorphism $\pi$ to composite with the \textbf{Hopf map} $\Phi$:%
\[%
\begin{array}
[c]{cccc}%
\pi: & \mathcal{H}_{2}^{0}(\mathbb{F}) & \rightarrow & \mathbb{R}%
\oplus\mathbb{F}\\
&
\begin{pmatrix}
\lambda & \alpha\\
\overline{\alpha} & -\lambda
\end{pmatrix}
& \mapsto & (\lambda,\alpha)
\end{array}
.
\]

Since the norms will be used in later, here, we briefly recall the norm
structures on the relevant spaces. The vector spaces $\mathbb{F}%
^{2},\mathcal{H}_{2}^{0}(\mathbb{F})$ and $\mathbb{R}\oplus\mathbb{F}$, which
appear in the definition of $\Phi$ and $\pi$, are equipped with the following
natural norms:
\[
|v|=\sqrt{v^{\ast}v}=\sqrt{tr(vv^{\ast})},\ |A|=\sqrt{\frac{1}{2}tr(A^{\ast
}A)}=\sqrt{\frac{1}{2}tr(A^{2})},\ |(\lambda,\alpha)|=\sqrt{|\lambda
|^{2}+|\alpha|^{2}}%
\]
for $v\in\mathbb{F}^{2},A\in\mathcal{H}_{2}^{0}(\mathbb{F})$ and
$(\lambda,\alpha)\in\mathbb{R}\oplus\mathbb{F}$. These norms are well-behaved
under the maps $\Phi$ and $\pi$, as present in the below Lemma.

\begin{lemma}
\label{norm} For $\mathbb{F}=\mathbb{R},\mathbb{C},\mathbb{H\ }$and the Hopf
map $\Phi$, the norm satisfies the identity $|\Phi(v)|=\frac{1}{2}|v|^{2}$.
Moreover, the isomorphism $\pi:\mathcal{H}_{n}^{0}(\mathbb{F})\rightarrow
\mathbb{R}\oplus\mathbb{F}$ is norm-preserving. i.e., $|A|=|\pi(A)|$ for all
$A\in\mathcal{H}_{n}^{0}(\mathbb{F})$.
\end{lemma}

\begin{proof}
For\ $v=\left(  x\ y\right)  ^{t}$, $|v|^{2}=\left\vert x\right\vert
^{2}+\left\vert y\right\vert ^{2}\ $and
\[
|\Phi(v)|=\sqrt{\frac{1}{2}tr(\Phi(v)^{2})}=\sqrt{\left(  \frac{|x|^{2}%
-|y|^{2}}{2}\right)  ^{2}+\left\vert x\overline{y}\right\vert ^{2}}%
=\frac{|x|^{2}+|y|^{2}}{2},
\]
thus $|\Phi(v)|=\frac{1}{2}|v|^{2}$. For $A\in\mathcal{H}_{n}^{0}(\mathbb{F}%
)$,$\ |A|=|\pi(A)|\ $is clear.\ \ \ $\blacksquare$
\end{proof}

\begin{lemma}
\label{surjection} For $\mathbb{F}=\mathbb{R},\mathbb{C},\mathbb{H}$,

(1) The Hopf map $\Phi$ is surjective.

(2) For each $(\lambda,\alpha)\in\mathbb{R}\oplus\mathbb{F}$,
\[%
\begin{cases}
\Phi^{-1}\left(
\begin{pmatrix}
\lambda & \alpha\\
\overline{\alpha} & -\lambda
\end{pmatrix}
\right)  \simeq\mathbb{F}(1),\;\text{for}\;(\lambda,\alpha)\neq(0,0)\\
\Phi^{-1}\left(
\begin{pmatrix}
0 & 0\\
0 & 0
\end{pmatrix}
\right)  =\left\{
\begin{pmatrix}
0\\
0
\end{pmatrix}
\right\}  ,\;\text{for}\;(\lambda,\alpha)=(0,0)
\end{cases}
\]

\end{lemma}

\begin{proof}
To prove the statement $\Phi\left(  \mathbb{F}^{2}\right)  =\mathcal{H}%
_{2}^{0}(\mathbb{F})$, it suffices to find, for any given $(\lambda,\alpha
)\in\mathbb{R}\oplus\mathbb{F}$, a pair $\left(  x,y\right)  \in\mathbb{F}%
^{2}\ $satisfying
\[%
\begin{cases}
x\overline{y}=\alpha\\
\frac{1}{2}(\left\vert x\right\vert ^{2}-\left\vert y\right\vert ^{2})=\lambda
\end{cases}
\]
This system is solvable, and in the process of solving it, we also obtain the
second statement. From these two equations, we derive the following identity:
\[
\left\vert x\right\vert ^{2}+\left\vert y\right\vert ^{2}=2\sqrt{\lambda
^{2}+\left\vert \alpha\right\vert ^{2}}.
\]
Substituting this back into the original equations yields
\[%
\begin{cases}
\left\vert x\right\vert ^{2}=\lambda+\sqrt{\lambda^{2}+\left\vert
\alpha\right\vert ^{2}}\\
\left\vert y\right\vert ^{2}=-\lambda+\sqrt{\lambda^{2}+\left\vert
\alpha\right\vert ^{2}}%
\end{cases}
\]

Now we consider two cases depending on whether $\alpha= 0$ or not.

Case 1: $\alpha\neq0$ ( i.e., $y\neq0$)

We set $y=\left\vert y\right\vert \theta=\sqrt{-\lambda+\sqrt{\lambda
^{2}+\left\vert \alpha\right\vert ^{2}}}\theta$ where $\theta\in\mathbb{F}%
(1)$. Then we have
\[
x=\left\vert x\right\vert \frac{\alpha}{\left\vert \alpha\right\vert }%
\theta=\sqrt{\lambda+\sqrt{\lambda^{2}+\left\vert \alpha\right\vert ^{2}}%
}\frac{\alpha}{\left\vert \alpha\right\vert }\theta
\]
Therefore, when $\alpha\neq0$, a solution is given by
\[
(x,y)=\left(  \sqrt{\lambda+\sqrt{\lambda^{2}+\left\vert \alpha\right\vert
^{2}}}\frac{\alpha}{\left\vert \alpha\right\vert }\theta,\sqrt{-\lambda
+\sqrt{\lambda^{2}+\left\vert \alpha\right\vert ^{2}}}\theta\right)
,\theta\in\mathbb{F}(1)
\]

Case 2: $\alpha=0$

In this case, the system is reduced to
\[%
\begin{cases}
x\overline{y}=0\\
\frac{1}{2}(\left\vert x\right\vert ^{2}-\left\vert y\right\vert ^{2})=\lambda
\end{cases}
\]
Solving this together with the identity
\[
\left\vert x\right\vert ^{2}+\left\vert y\right\vert ^{2}=2\sqrt{\lambda^{2}%
}=2|\lambda|,
\]
we obtain
\[%
\begin{cases}
\left\vert x\right\vert =\lambda+|\lambda|\\
\left\vert y\right\vert =-\lambda+|\lambda|.
\end{cases}
\]
The solutions are given as\ for$\ \theta\in\mathbb{F}(1)$
\[
(x,y)=\left\{
\begin{array}
[c]{c}%
\left(  2\lambda\theta,0\right)  \ \text{if }\lambda>0\text{ }\\
\left(  0,0\right)  \ \ \text{if }\lambda=0\\
\left(  0,-2\lambda\theta\right)  \ \text{if }\lambda<0\text{ }%
\end{array}
\right.  \text{for}\ \theta\in\mathbb{F}(1).
\]

This proves the lemma.\ \ \ \ $\blacksquare$
\end{proof}

The Lemma \ref{surjection}\ implies that for each pair of $\left(  x,y\right)
,\left(  a,b\right)  \in\mathbb{F}^{2}$, if such that $\Phi\left(  \left(
x,y\right)  \right)  =\Phi\left(  \left(  a,b\right)  \right)  \ $then
\[
\left\vert x\right\vert =\left\vert a\right\vert \text{, }\left\vert
y\right\vert =\left\vert b\right\vert \text{,}%
\]
and there is $c\in\mathbb{F}(1)$ such that $\left(  x,y\right)  =\left(
ac,bc\right)  $. In fact, by considering a right group action of
$\mathbb{F}(1)$ on $\mathbb{F}^{2}$, defined by%
\[%
\begin{array}
[c]{cccc}%
\cdot: & \mathbb{F}^{2}\times\mathbb{F}(1) & \rightarrow & \mathbb{F}^{2}\\
& \left(
\begin{pmatrix}
x\\
y
\end{pmatrix}
,c\right)  & \mapsto &
\begin{pmatrix}
x\\
y
\end{pmatrix}
\cdot c:=%
\begin{pmatrix}
xc\\
yc
\end{pmatrix}
,
\end{array}
\]
the group action preserves the fibers of the Hopf map\ $\Phi$:
\[
\Phi\left(
\begin{pmatrix}
x\\
y
\end{pmatrix}
\cdot c\right)  =\Phi\left(
\begin{pmatrix}
x\\
y
\end{pmatrix}
\right)
\]
since $\left\vert xc\right\vert =\left\vert x\right\vert \,$, $\left\vert
yc\right\vert =\left\vert y\right\vert \ $and $\left(  xc\right)  \left(
\overline{yc}\right)  =x\left(  c\bar{c}\right)  y=xy$. Note we need the
associative condition in the last one.

From the $\mathbb{F}(1)$-action on $\mathbb{F}^{2}$, we consider the orbits
$\mathbb{F}^{2}/\mathbb{F}(1)\ $and related map $p:\mathbb{F}^{2}%
\rightarrow\mathbb{F}^{2}/\mathbb{F}(1)$ and the following commutative diagram
to define a map $\widetilde{\Phi}$ satisfying $\Phi=\widetilde{\Phi}\circ p$.%

\[%
\begin{matrix}
\Phi: & \mathbb{F}^{2}\smallsetminus\left\{  0\right\}  & \longrightarrow &
\mathcal{H}_{2}^{0}(\mathbb{F})\smallsetminus\left\{  0\right\} \\
&  &  & \\
& p\big\downarrow & \nearrow & \widetilde{\Phi}\\
&  &  & \\
& \mathbb{F}^{2}\smallsetminus\left\{  0\right\}  /\mathbb{F}(1) &  &
\end{matrix}
\]
Here, $\widetilde{\Phi}$ is well defined and a bijection between
$\mathbb{F}^{2}\smallsetminus\left\{  0\right\}  /\mathbb{F}(1)\simeq$
$\mathbb{FP}^{1}$and $\mathcal{H}_{2}^{0}(\mathbb{F})\smallsetminus\left\{
0\right\}  $.

\subsection{Hopf maps and $\mathbb{O}$}

As $\mathbb{F}=\mathbb{R},\mathbb{C},\mathbb{H}$, we also define Hopf map
$\Phi_{\mathbb{O}}\ $for octonions $\mathbb{O}$:%
\[%
\begin{array}
[c]{cccc}%
\Phi: & \mathbb{O}^{2} & \rightarrow & \mathcal{H}_{2}^{0}(\mathbb{O})\\
& v & \mapsto & vv^{\ast}-\frac{tr(vv^{\ast})}{2}I_{2}%
\end{array}
\]
where $\mathcal{H}_{2}^{0}(\mathbb{O})$ denotes the space of $2\times2$
hermitian traceless matrices over $\mathbb{O\ }$and $v\ $is considered as a
column vector. Explicitly, the image is given by:
\[
\Phi_{\mathbb{O}}\left(
\begin{pmatrix}
x\\
y
\end{pmatrix}
\right)  =%
\begin{pmatrix}
\frac{|x|^{2}-|y|^{2}}{2} & x\bar{y}\\
y\bar{x} & \frac{|y|^{2}-|x|^{2}}{2}%
\end{pmatrix}
.
\]

The isomorphism $\pi\ $is also defined to composite with the \textbf{Hopf map
}$\Phi:$
\[%
\begin{array}
[c]{cccc}%
\pi: & \mathcal{H}_{2}^{0}(\mathbb{O}) & \rightarrow & \mathbb{R}%
\oplus\mathbb{O}\\
&
\begin{pmatrix}
\lambda & \alpha\\
\overline{\alpha} & -\lambda
\end{pmatrix}
& \mapsto & (\lambda,\alpha)
\end{array}
\]
And we consider
\[
\mathbb{O}(1)=\{c\in\mathbb{O}:|c|=1\}=S^{7}%
\]
which is not a group but a Moufang loop as it satisfies axioms of the group
except the associativity.

By performing similar argument to Lemma \ref{surjection}, we have the
following Lemma for $\mathbb{O}$.

\begin{lemma}
\label{SurjectivOctonions}For octonions $\mathbb{O}$,

(1) The Hopf map $\Phi_{\mathbb{O}}$ is surjective.

(2) For each $(\lambda,\alpha)\in\mathbb{R}\oplus\mathbb{O}$,
\[%
\begin{cases}
\Phi_{\mathbb{O}}^{-1}\left(
\begin{pmatrix}
\lambda & \alpha\\
\overline{\alpha} & -\lambda
\end{pmatrix}
\right)  \simeq\mathbb{O}(1),\;\text{for}\;(\lambda,\alpha)\neq(0,0)\\
\Phi_{\mathbb{O}}^{-1}\left(
\begin{pmatrix}
0 & 0\\
0 & 0
\end{pmatrix}
\right)  =\left\{
\begin{pmatrix}
0\\
0
\end{pmatrix}
\right\}  ,\;\text{for}\;(\lambda,\alpha)=(0,0)
\end{cases}
\]

\end{lemma}

\begin{remark}
Here we solve%
\[%
\begin{cases}
x\bar{y}=\alpha\\
\frac{1}{2}(\left\vert x\right\vert ^{2}-\left\vert y\right\vert ^{2})=\lambda
\end{cases}
\]
for octonions $\mathbb{O}$. When $\alpha\neq0\,$, the solutions of the system
can be written
\[
(x,y)=(\left\vert x\right\vert \left\vert \alpha\right\vert ^{-1}\alpha
\theta,\left\vert y\right\vert \theta)\text{,\ }\theta\in\mathbb{O}(1)
\]
and one may consider typical $\mathbb{O}(1)$-multiplication to $\Phi
_{\mathbb{O}}^{-1}\left(
\begin{pmatrix}
\lambda & \alpha\\
\overline{\alpha} & -\lambda
\end{pmatrix}
\right)  \ $to the right as
\[
(xc,yc)=(\left\vert x\right\vert \left\vert \alpha\right\vert ^{-1}\left(
\alpha\theta\right)  c,\left\vert y\right\vert \theta c)
\]
for $c\in\mathbb{O}(1)$. But there is an issue of associativity to get
$\left\vert x\right\vert \left\vert \alpha\right\vert ^{-1}\alpha\left(
c\theta\right)  \ $with respect to $\theta\in\mathbb{O}(1)\ $related to
$\Phi_{\mathbb{O}}^{-1}\left(
\begin{pmatrix}
\lambda & \alpha\\
\overline{\alpha} & -\lambda
\end{pmatrix}
\right)  $. For this reason, we introduce an alternative $\mathbb{O}%
(1)$-multiplication to $\mathbb{O}^{2}$.
\end{remark}

\bigskip

\begin{center}
\textbf{ Multiplication of }$\mathbb{O}(1)\ $\textbf{to }$\mathbb{O}^{2}$
\end{center}

Even though $\mathbb{O}(1)\ $is not a group but a Moufang loop, we define
$\mathbb{O}(1)$-multiplication to $\mathbb{O}^{2}\ $%

\[%
\begin{array}
[c]{cccc}%
\cdot: & \mathbb{O}^{2}\times\mathbb{O}(1) & \rightarrow & \mathbb{O}^{2}\\
& \left(
\begin{pmatrix}
x\\
y
\end{pmatrix}
,c\right)  & \mapsto &
\begin{pmatrix}
x\\
y
\end{pmatrix}
\cdot c:=\left\{
\begin{array}
[c]{c}%
\left(
\begin{array}
[c]{c}%
\left(  xy^{-1}\right)  \left(  yc\right) \\
yc
\end{array}
\right)  \ \text{if }y\neq0\\
\left(
\begin{array}
[c]{c}%
xc\\
0
\end{array}
\right)  \ \text{if }y=0
\end{array}
\right.  .
\end{array}
\]
Here $\left(  xy^{-1}\right)  \left(  yc\right)  =xc\ $under the associative
cases. Note that the $\mathbb{O}(1)$-multiplication to $\mathbb{O}^{2}\ $is to
the left unlike $\mathbb{F}=\mathbb{R},\mathbb{C},\mathbb{H\ }$and different
definition with \cite{Or} to study the Hopf fibration for $\mathbb{O}$. Now
the $\mathbb{O}(1)$-multiplication to $\mathbb{O}^{2}\ $is compatible to
$\mathbb{O}(1)$-family $\Phi_{\mathbb{O}}^{-1}\left(
\begin{pmatrix}
\lambda & \alpha\\
\overline{\alpha} & -\lambda
\end{pmatrix}
\right)  \ $in the Lemma \ref{SurjectivOctonions}\ in the following Lemma.

\begin{lemma}
\label{OctonionFiberLemma} For $%
\begin{pmatrix}
x\\
y
\end{pmatrix}
,%
\begin{pmatrix}
a\\
b
\end{pmatrix}
$ $\in\mathbb{O}^{2}\smallsetminus\left\{
\begin{pmatrix}
0\\
0
\end{pmatrix}
\right\}  $,
\[
\Phi_{\mathbb{O}}\left(
\begin{pmatrix}
x\\
y
\end{pmatrix}
\right)  =\Phi_{\mathbb{O}}\left(
\begin{pmatrix}
a\\
b
\end{pmatrix}
\right)  \ \text{if and only if there is\ }c\in\mathbb{O}(1)\mathbb{\ }%
\text{such that }%
\begin{pmatrix}
a\\
b
\end{pmatrix}
=%
\begin{pmatrix}
x\\
y
\end{pmatrix}
\cdot c\text{.}%
\]

\end{lemma}

\begin{proof}
Assume there is $c\in\mathbb{O}(1)\mathbb{\ }$such that $%
\begin{pmatrix}
a\\
b
\end{pmatrix}
=%
\begin{pmatrix}
x\\
y
\end{pmatrix}
\cdot c$, we check%
\[
\left\vert a\right\vert ^{2}-\left\vert b\right\vert ^{2}=\left\vert
x\right\vert ^{2}-\left\vert y\right\vert ^{2},a\bar{b}=x\bar{y}\text{.}%
\]
If $\left\vert y\right\vert \neq0$,
\[
\text{ }%
\begin{pmatrix}
a\\
b
\end{pmatrix}
=%
\begin{pmatrix}
\left(  xy^{-1}\right)  \left(  yc\right) \\
yc
\end{pmatrix}
\]
and
\begin{align*}
\left\vert a\right\vert ^{2}-\left\vert b\right\vert ^{2}  &  =\left\vert
\left(  xy^{-1}\right)  \left(  yc\right)  \right\vert ^{2}-\left\vert
yc\right\vert ^{2}\\
&  =\left(  \left\vert x\right\vert \left\vert y\right\vert ^{-1}\left\vert
y\right\vert \left\vert c\right\vert \right)  ^{2}-\left(  \left\vert
y\right\vert \left\vert c\right\vert \right)  ^{2}=\left\vert x\right\vert
^{2}-\left\vert y\right\vert ^{2}\\
a\bar{b}  &  =\left(  \left(  xy^{-1}\right)  \left(  yc\right)  \right)
\overline{\left(  yc\right)  }=\left(  xy^{-1}\right)  \left\vert
yc\right\vert ^{2}\ \text{by Lemma \ref{ConjugationLemma2}}\\
&  =x\bar{y}^{-1}\left\vert y\right\vert ^{2}\left\vert c\right\vert
^{2}=x\bar{y}.
\end{align*}
Similarly, case $\left\vert y\right\vert =0\ $is clear.

Assume $\Phi_{\mathbb{O}}\left(
\begin{pmatrix}
x\\
y
\end{pmatrix}
\right)  =\Phi_{\mathbb{O}}\left(
\begin{pmatrix}
a\\
b
\end{pmatrix}
\right)  $ then as the argument of Lemma \ref{SurjectivOctonions}, we have
\[
\left\vert a\right\vert =\left\vert x\right\vert ,\left\vert b\right\vert
=\left\vert y\right\vert .
\]
If $\left\vert y\right\vert \neq0$, there is $c\in\mathbb{O}(1)~$such that
$b=yc$. Moreover,
\[
a=\alpha b\left\vert b\right\vert ^{-2}=\left(  x\bar{y}\right)  \left(
yc\right)  \left\vert yc\right\vert ^{-2}=\left(  \left\vert y\right\vert
^{-2}x\bar{y}\right)  \left(  yc\right)  =\left(  xy^{-1}\right)  \left(
yc\right)  \text{.}%
\]
Similarly, case $\left\vert y\right\vert =0\ $is clear.\ \ \ $\blacksquare$
\end{proof}

The Lemma \ref{OctonionFiberLemma} implies that $\Phi_{\mathbb{O}}^{-1}\left(
\begin{pmatrix}
\lambda & \alpha\\
\overline{\alpha} & -\lambda
\end{pmatrix}
\right)  \ $can be considered as an orbit of $\mathbb{O}(1)$-multiplication to
$\mathbb{O}^{2}$. Thus we also consider the orbits $\mathbb{O}^{2}%
/\mathbb{O}(1)\ $and related map $p:\mathbb{O}^{2}\rightarrow\mathbb{O}%
^{2}/\mathbb{O}(1)$ and the following commutative diagram to define a map
$\widetilde{\Phi}_{\mathbb{O}}$ satisfying $\Phi_{\mathbb{O}}=\widetilde{\Phi
}_{\mathbb{O}}\circ p$.%

\[%
\begin{matrix}
\Phi_{\mathbb{O}}: & \mathbb{O}^{2}\smallsetminus\left\{  0\right\}  &
\longrightarrow & \mathcal{H}_{2}^{0}(\mathbb{O})\smallsetminus\left\{
0\right\} \\
&  &  & \\
& p\ \big\downarrow & \nearrow & \widetilde{\Phi}_{\mathbb{O}}\\
&  &  & \\
& \mathbb{O}^{2}\smallsetminus\left\{  0\right\}  /\mathbb{O}(1) &  &
\end{matrix}
\]
Here, $\widetilde{\Phi}$ is well defined and a bijection between
$\mathbb{O}^{2}\smallsetminus\left\{  0\right\}  /\mathbb{O}(1)\simeq$
$\mathbb{OP}^{1}$and $\mathcal{H}_{2}^{0}(\mathbb{O})\smallsetminus\left\{
0\right\}  $.

\bigskip

\section{Spin representation for Hopf maps}

\subsection{Special Unitary group $SU(2,\mathbb{F})\ $for $\mathbb{F}%
=\mathbb{R},\mathbb{C},\mathbb{H}$}

In this subsection, we consider $\mathbb{F}=\mathbb{R},\mathbb{C},\mathbb{H}$.
We extend the standard definition of the special unitary group $SU(2)$ to the
cases $\mathbb{F}=\mathbb{R},\mathbb{C},\mathbb{H}$ by defining
\[
SU(2,\mathbb{F})=\{A\in Mat_{2}(\mathbb{F}):AA^{\ast}=I=A^{\ast}A,\det A=1\}.
\]
In the case of quaternions $\mathbb{F}=\mathbb{H}$, the unitary condition
implies that $\det A=1$. Therefore, the corresponding groups are:
\[%
\begin{tabular}
[c]{|c||c|c|c|}\hline
$\mathbb{F}$ & $\mathbb{R}$ & $\mathbb{C}$ & $\mathbb{H}$\\\hline
$SU(2,\mathbb{F})$ & $SO(2)$ & $SU(2)$ & $Sp(2)$\\\hline
\end{tabular}
\]

It is well known that these Lie groups are isomorphic to the corresponding
real spin groups. Let us recall the real spin groups $Spin(n)$ in low
dimensions $n=2,3,5$. Then, the groups $Spin(n)$ are isomorphic to
$SU(2,\mathbb{F})$ for $\mathbb{F}=\mathbb{R},\mathbb{C},\mathbb{H}$,
respectively. Moreover, their irreducible real spin representations are given
by $\mathbb{F}^{2}$, known as the spinor spaces. In these cases, the spin
action can be explicitly realized via the standard representation of
$SU(2,\mathbb{F})$. Such a direct realization of the spin representation is
usually not available in general dimensions. These are corresponded to another
group action of $SU(2,\mathbb{F})$, which is the adjoint action on
$\mathcal{H}_{2}^{0}(\mathbb{F})$, defined by
\[
Ad_{A}(X)=AXA^{\ast}%
\]
where $A\in SU(2,\mathbb{F})$ and $X\in\mathcal{H}_{2}^{0}(\mathbb{F})$. This
action is obviously well-defined for $\mathbb{R}$ and $\mathbb{C}$, but one
must be cautious in the quaternionic case. Since conjugate transpose property
still holds for $\mathbb{H}$, $Ad_{A}(X)$ is hermitian.
\[
(Ad_{A}(X))^{\ast}=(AXA^{\ast})^{\ast}=AX^{\ast}A^{\ast}=AXA^{\ast}=Ad_{A}(X)
\]

\begin{lemma}
\label{HtoC} There is a group homomorphism $h:M_{n}(\mathbb{H})\rightarrow
M_{2n}(\mathbb{C})$ satisfying
\[
h(A^{\ast})=h(A)^{\ast},\quad\text{and}\quad tr\left(  h(A)\right)
=2\operatorname{Re}tr(A).
\]

\end{lemma}

\begin{proof}
We consider quaternionic matrix $A\in M_{n}(\mathbb{H})\ $as
\[
A=A_{1}+A_{2}j
\]
for $A_{1}$,$\ A_{2}\in M_{n}(\mathbb{C})$ with respect to the identification
$q=z_{1}+z_{2}j\in\mathbb{H\ }$for $z_{1},z_{2}\in\mathbb{C}$. And we define
the group homomorphism
\[%
\begin{array}
[c]{cccc}%
h: & M_{n}(\mathbb{H}) & \rightarrow & M_{2n}(\mathbb{C})\\
& A & \longmapsto & \left(
\begin{array}
[c]{cc}%
A_{1} & A_{2}\\
-\bar{A}_{2} & \bar{A}_{1}%
\end{array}
\right)  .
\end{array}
\]
Then
\[
h(A^{\ast})=h(A_{1}^{\ast}-A_{2}^{t}j)=\left(
\begin{array}
[c]{cc}%
A_{1}^{\ast} & -A_{2}^{t}\\
\bar{A}_{2}^{t} & \overline{\left(  A_{1}^{\ast}\right)  }%
\end{array}
\right)  =\left(
\begin{array}
[c]{cc}%
A_{1} & A_{2}\\
-\bar{A}_{2} & \overline{A_{1}}%
\end{array}
\right)  ^{\ast}=h(A)^{\ast}\text{,}%
\]
and
\[
tr\left(  h(A)\right)  =tr\left(
\begin{array}
[c]{cc}%
A_{1} & A_{2}\\
-\bar{A}_{2} & \bar{A}_{1}%
\end{array}
\right)  =\left(  trA_{1}+tr\bar{A}_{1}\right)  =2\operatorname{Re}tr(A).
\]

\ $\blacksquare$
\end{proof}

\bigskip

Although the cyclic property of the trace does not hold in general for
quaternionic matrices due to non-commutativity, we can verify the invariance
of the trace under the adjoint action via the homomorphism $h$, using the
cyclicity of the trace in the complex setting.

\begin{lemma}
The trace on $\mathcal{H}_{2}(\mathbb{H})$ is invariant under the adjoint
action of $SU(2,\mathbb{H})$, i.e.,
\[
tr(AXA^{\ast})=tr(X),\quad\text{for all }A\in SU(2,\mathbb{H}),X\in
\mathcal{H}_{2}^{0}(\mathbb{H}).
\]

\end{lemma}

\begin{proof}
Using Lemma \ref{HtoC} and the fact that $AXA^{\ast}$is\ Hermitian$~$for
$X\in\mathcal{H}_{2}(\mathbb{H})\ $and $A\in M_{n}(\mathbb{H})$, we compute:
\begin{align*}
tr(AXA^{\ast})  &  =\operatorname{Re}(tr(AXA^{\ast}))=\tfrac{1}{2}%
tr(h(AXA^{\ast}))=\tfrac{1}{2}tr(h(A)h(X)h(A^{\ast}))\\
&  =\tfrac{1}{2}tr(h(X))=\operatorname{Re}(tr(X))=tr(X)
\end{align*}
as required.\ \ \ \ \ $\blacksquare$
\end{proof}

Therefore, we have verified that the adjoint action of $SU(2,\mathbb{F})$ on
$\mathcal{H}_{2}^{0}(\mathbb{F})$ is well-defined for all $\mathbb{F}%
=\mathbb{R},\mathbb{C},\mathbb{H}$. In summary, we have the followings. Later
of this section, we also extend the table for$\ \mathbb{F}=\mathbb{R}%
,\mathbb{C},\mathbb{H\ }$to the octonions $\mathbb{O}$ by considering
$SU(2,\mathbb{O})\ $as $Spin\left(  9\right)  \ $and related spinors. In fact
the extension to $\mathbb{O\ }$can be considered for $Spin\left(  8\right)  $
to $\mathbb{O}^{2}$ along the triality. We discuss this in another article.%

\[
\underset{\text{\textbf{Table1: Special\ Unitary\ Groups\ and Spin Groups}}}{%
\begin{tabular}
[c]{|c|c|c|c|}\hline
$\mathbb{F}$ & $SU(2,\mathbb{F})\simeq Spin\left(  n\right)  $ & Spinors &
$\mathcal{H}_{2}^{0}(\mathbb{F})\simeq\mathbb{R}\oplus\mathbb{F}$, $\left(
\dim\right)  $\\\hline\hline
$\mathbb{R}$ & $SO(2)\simeq Spin\left(  2\right)  $ & $\mathbb{R}^{2}$ &
$\mathbb{R}\oplus\mathbb{R}$,$\mathbb{\ (}2\mathbb{)}$\\\hline
$\mathbb{C}$ & $SU(2)\simeq Spin\left(  3\right)  $ & $\mathbb{C}^{2}$ &
$\mathbb{R}\oplus\mathbb{C}$, $\mathbb{(}3\mathbb{)}$\\\hline
$\mathbb{H}$ & $Sp(2)\simeq Spin\left(  5\right)  $ & $\mathbb{H}^{2}$ &
$\mathbb{R}\oplus\mathbb{H}$,$\ \mathbb{(}5\mathbb{)}$\\\hline
$\mathbb{O}$ & $Spin\left(  9\right)  $ & $\mathbb{O}^{2}$ & $\mathbb{R}%
\oplus\mathbb{O}$,$\ \mathbb{(}9\mathbb{)}$\\\hline
\end{tabular}
\ }%
\]

\bigskip

\subsection{Group action of $SU(2,\mathbb{F}),\mathbb{F}=\mathbb{R}%
,\mathbb{C},\mathbb{H}$ and Hopf maps}

\begin{proposition}
\label{SORCH} For $\mathbb{F}=\mathbb{R},\mathbb{C},\mathbb{H}$, the following hold:

(1) The Hopf map $\Phi:\mathbb{F}^{2}\rightarrow\mathcal{H}_{2}^{0}%
(\mathbb{F})$ is equivalent with respect to the $SU(2,\mathbb{F})$ actions.

(2) For the composition map $\pi\circ\Phi:\mathbb{F}^{2}\rightarrow
\mathbb{R}\oplus\mathbb{F}$, the standard action of $SU(2,\mathbb{F})$ induces
$SO(1+dim_{\mathbb{R}}\mathbb{F})$ action on $\mathbb{R}\oplus\mathbb{F}$.
\end{proposition}

\begin{proof}
Let $A\in SU(2,\mathbb{F})$ and $v=%
\begin{pmatrix}
x\\
y
\end{pmatrix}
\in\mathbb{F}^{2}$.

(1) Then
\begin{align*}
\Phi\left(  Av\right)   &  =\left(  Av\right)  \left(  Av\right)  ^{\ast
}-\frac{1}{2}tr\left(  \left(  Av\right)  \left(  Av\right)  ^{\ast}\right)
I_{2}=\left(  Av\right)  \left(  v^{\ast}A^{\ast}\right)  -\frac{1}%
{2}\left\vert Av\right\vert ^{2}I_{2}\\
&  =A\left(  vv^{\ast}\right)  A^{\ast}-\frac{1}{2}\left\vert v\right\vert
^{2}I_{2}=A\left(  vv^{\ast}-\frac{1}{2}\left\vert v\right\vert ^{2}%
I_{2}\right)  A^{\ast}\\
&  =A\Phi\left(  v\right)  A^{\ast}=Ad_{A}\left(  \Phi\left(  v\right)
\right)
\end{align*}
The second equality follows from the compatibility of the norm. The third
equality is because of $A\in SU(2,\mathbb{F})$ and that the three normed
division algebras $\mathbb{F}$ are associative. Therefore, $\Phi$ is
equivalent under the $SU(2,\mathbb{F})$ action.

(2) The standard action of $SU(2,\mathbb{F})$ on $\mathbb{F}^{2}$, followed by
$\pi\circ\Phi$, preserves the norm on $\mathbb{R}\oplus\mathbb{F}$. That is,
\[
\left\vert \pi\circ\Phi\left(  Av\right)  \right\vert =\left\vert \pi\circ
\Phi\left(  v\right)  \right\vert
\]
This implies that the induced action on $\mathbb{R}\oplus\mathbb{F}$ lies in
the orthogonal group $O(1+dim_{\mathbb{R}}\mathbb{F})$. Since the isomorphism
$\pi$ is norm-preserving, as shown in Lemma \ref{norm}, it suffices to verify
that
\[
\left\vert \Phi\left(  Av\right)  \right\vert =\left\vert \Phi\left(
v\right)  \right\vert .
\]
By Lemma \ref{norm}, we have
\[
\left\vert \Phi\left(  Av\right)  \right\vert =\frac{1}{2}\left\vert
Av\right\vert ^{2}%
\]
and since the standard action of $SU(2,\mathbb{F})$ preserve the norm on
$\mathbb{F}^{2}$, we conclude that
\[
\frac{1}{2}\left\vert Av\right\vert ^{2}=\frac{1}{2}\left\vert v\right\vert
^{2}\Rightarrow\left\vert \Phi\left(  Av\right)  \right\vert =\left\vert
\Phi\left(  v\right)  \right\vert .
\]

Furthermore, since $SU(2,\mathbb{F})$ is path-connected for all $\mathbb{F}%
=\mathbb{R},\mathbb{C},\mathbb{H}$, the image of the induced action lies not
only in $O(1+dim_{\mathbb{R}}\mathbb{F})$, but actually in the special
orthogonal group $SO(1+dim_{\mathbb{R}}\mathbb{F})$%
.\ $\ \ \ \ \ \ \blacksquare$
\end{proof}

\bigskip

Now for $\mathbb{F}=\mathbb{R},\mathbb{C},\mathbb{H}$, we consider two group
actions on $\mathbb{F}^{2}$, the standard action of $SU(2,\mathbb{F})$ and
right $\mathbb{F}(1)$ action. We examine the elements of these groups that act
compatibly under both. This is equivalent to finding the elements $c$ in the
center $Z(\mathbb{F})$ of $\mathbb{F}$ satisfying (i) $|c|=1$, so that
\[
c\in Z(\mathbb{F})\cap\mathbb{F}(1),
\]
and (ii) $cI_{2}\in SU(2,\mathbb{F})$.

In the cases of $\mathbb{F}=\mathbb{R},\mathbb{C}$, since $\mathbb{R}%
,\mathbb{C}$ are commutative, we have $Z(\mathbb{F})\cap\mathbb{F}%
(1)=\mathbb{F}(1)$. From the matrix group perspective, when $\mathbb{F}%
=\mathbb{R}$, the matrix \ $cI_{2}\in SU(2,\mathbb{R})=SO(2)$ if and only if
$c\in\{\pm1\}$. A similar argument holds for $\mathbb{F}=\mathbb{C}$, leading
again to $c\in\{\pm1\}$. On the other hand, when $\mathbb{F}=\mathbb{H}$,
commutativity fails, and the center is $Z(\mathbb{H})=\mathbb{R}$. Hence,
$Z(\mathbb{H})\cap\mathbb{H}(1)=\{\pm1\}$. For the matrix $cI_{2}$ to belong
to $SU(2,\mathbb{H})$, it must satisfy $|c|^{2}=1$. Therefore, in this case as
well, we conclude that $c\in\{\pm1\}$. Therefore, the standard action of
$SU(2,\mathbb{F})$ and the right action of $\mathbb{F}(1)$ act on
$\mathbb{F}^{2}$, the only elements that act compatibly on both sides are $\pm
I_{2}$.

As a result, we obtain the following diagram illustrating the group actions
and the Hopf map in this section for $\mathbb{F}=\mathbb{R},\mathbb{C}%
,\mathbb{H}$:%
\[
\underset{\text{\textbf{Diagram\ A}}}{%
\begin{array}
[c]{ccccc}%
SU(2,\mathbb{F}) & \curvearrowright & \mathbb{F}^{2} & \curvearrowleft &
\mathbb{F}(1)\\
\tilde{\pi}\downarrow &  & \Phi\downarrow &  & \downarrow\\
SU(2,\mathbb{F})/\left\{  \pm I_{2}\right\}  & \curvearrowright &
\mathcal{H}_{2}^{0}(\mathbb{F}) & \curvearrowleft & \left\{  1\right\} \\
\tilde{\rho}\downarrow &  & \pi\downarrow &  & \\
SO(1+dim_{\mathbb{R}}\mathbb{F}) & \curvearrowright & \mathbb{R}%
\oplus\mathbb{F} &  &
\end{array}
}%
\]
Here $SU(2,\mathbb{F})$-action to $\mathbb{F}^{2}$ is the spin action to
spinor $\mathbb{F}^{2}$ and $SU(2,\mathbb{F})/\left\{  \pm I_{2}\right\}
\ $action to $\mathcal{H}_{2}^{0}(\mathbb{F})\ $is the corresponding
$SO(1+dim_{\mathbb{R}}\mathbb{F})$ action on $\pi\left(  \mathcal{H}_{2}%
^{0}(\mathbb{F})\right)  =\mathbb{R}\oplus\mathbb{F\ }$so that $\tilde{\rho
}\circ\tilde{\pi}$ is\ $2$-covering of $Spin\left(  1+dim_{\mathbb{R}%
}\mathbb{F}\right)  $ for $\mathbb{F}=\mathbb{R},\mathbb{C},\mathbb{H}$.

\bigskip

\subsection{ \label{subsectionOandGroupaction}$SU(2,\mathbb{O})\ $and Hopf map
for $\mathbb{O}$}

In this subsection, we want to construct octonionic version of Diagram\ A in
the previous subsection. For this purpose, we consider the extension of
special unitary groups $SU(2,\mathbb{F})\ $of $\mathbb{F}=\mathbb{R}%
,\mathbb{C},\mathbb{H}$ in Table 1\ to octonions $\mathbb{O\ }$as
$SU(2,\mathbb{O})=Spin(9)$ along the Hopf maps. Recall that $\mathbb{O}\left(
1\right)  $-multiplication to $\mathbb{O}^{2}\ $is to the right. We
consider\ a model for Clifford algebra of a $9$-dimensional vector space from
\cite{Har} (chapter 14) to define $Spin(9)$.

\begin{lemma}
\label{Spin9Lemma}( \cite{Har}\ Lemma 14.77) The group $Spin(9)\ $is generated
by the $8$-sphere
\[
\left\{  \left.  \left(
\begin{array}
[c]{cc}%
r & R_{u}\\
R_{\bar{u}} & -r
\end{array}
\right)  \right\vert \ r\in\mathbb{R}\text{,}\ u\in\mathbb{O}\text{, and
}r^{2}+\left\vert u\right\vert ^{2}=1\text{ }\right\}
\]
in $End_{\mathbb{R}}\left(  \mathbb{O\oplus O}\right)  \simeq Cl\left(
9\right)  ^{even}$.
\end{lemma}

\begin{remark}
Here $R_{u}\left(  w\right)  :=wu$ for $w\in\mathbb{O}$ and $Cl\left(
9\right)  $ is the Clifford algebra by a $9$-dimensional vector space$\ $and
$Cl\left(  9\right)  ^{even}\ $is the subalgebra of $Cl\left(  9\right)
\ $consisted of even degree elements. In \cite{Har}\ (Theorem 14.99), we can
also consider the copy
\[
\widehat{Spin(9)}:=C\cdot Spin(9)\cdot C
\]
$\ $of $Spin(9)\ $where $C\ $is the conjugation on the both factors of
$\mathbb{O\oplus O\,}$. Since
\[
C\cdot\left(
\begin{array}
[c]{cc}%
r & R_{u}\\
R_{\bar{u}} & -r
\end{array}
\right)  \cdot C=\left(
\begin{array}
[c]{cc}%
r & L_{\bar{u}}\\
L_{u} & -r
\end{array}
\right)
\]
where $L_{u}\left(  w\right)  :=uw$ for $w\in\mathbb{O}$, $\widehat
{Spin(9)}\ $is generated by
\[
\left\{  \left.  \left(
\begin{array}
[c]{cc}%
r & L_{\bar{u}}\\
L_{u} & -r
\end{array}
\right)  \right\vert \ r\in\mathbb{R}\text{,}\ u\in\mathbb{O}\text{, and
}r^{2}+\left\vert u\right\vert ^{2}=1\text{ }\right\}
\]
by the Lemma \ref{Spin9Lemma}.
\end{remark}

\bigskip

The non-associativity of octonions $\mathbb{O}$ is the biggest issue to
construct the octonionic version of Diagram\ A. For this issue, we consider
the generators given by unit $8$-sphere
\[
\mathbb{R\oplus O}\left(  1\right)  :=\left\{  \left.  \left(  r,u\right)
\in\mathbb{R\oplus O}\right\vert \ \text{ }r^{2}+\left\vert u\right\vert
^{2}=1\text{ }\right\}
\]
and related map defined by%
\[%
\begin{array}
[c]{cccc}%
g: & \mathbb{R\oplus O}\left(  1\right)  & \rightarrow & \widehat
{Spin(9)}\subset End_{\mathbb{R}}\left(  \mathbb{O\oplus O}\right) \\
& \left(  r,u\right)  & \mapsto & g\left(  r,u\right)  :=\left(
\begin{array}
[c]{cc}%
r & L_{\bar{u}}\\
L_{u} & -r
\end{array}
\right)  .
\end{array}
\]
Note $g\left(  r,u\right)  g\left(  r,u\right)  =\left(  r^{2}+\left\vert
u\right\vert ^{2}\right)  I_{2}=I_{2}$ and
\[
g\left(  r,u\right)  v=%
\begin{pmatrix}
rx+\bar{u}y\\
ux-ry
\end{pmatrix}
\ \text{for }v=%
\begin{pmatrix}
x\\
y
\end{pmatrix}
.
\]
We compute $\pi\Phi\left(  g\left(  r,u\right)  v\right)  \in\mathbb{R\oplus
O\ }$in the following lemma which is useful in the remain part of the article.

\begin{lemma}
\label{HopfOctKeyLemma}For $\left(  r,u\right)  \in\mathbb{R\oplus O}\left(
1\right)  $ and $x,y\in\mathbb{O}$,%
\begin{align*}
\frac{1}{2}\left(  \left\vert rx+\bar{u}y\right\vert ^{2}-\left\vert
ux-ry\right\vert ^{2}\right)   &  =\left(  r^{2}-\left\vert u\right\vert
^{2}\right)  \left(  \frac{\left\vert x\right\vert ^{2}-\left\vert
y\right\vert ^{2}}{2}\right)  +2r\left\langle \bar{u},x\bar{y}\right\rangle \\
\left(  rx+\bar{u}y\right)  \overline{\left(  ux-ry\right)  }  &  =2r\bar
{u}\left(  \frac{\left\vert x\right\vert ^{2}-\left\vert y\right\vert ^{2}}%
{2}\right)  -x\bar{y}+2\bar{u}\left\langle \bar{u},x\bar{y}\right\rangle
\end{align*}

\end{lemma}

\begin{proof}
For the first one,%
\begin{align*}
&  \frac{1}{2}\left(  \left\vert rx+\bar{u}y\right\vert ^{2}-\left\vert
ux-ry\right\vert ^{2}\right) \\
&  =\frac{1}{2}\left(
\begin{array}
[c]{c}%
r^{2}\left\vert x\right\vert ^{2}+rx\overline{\left(  \bar{u}y\right)
}+r\left(  \bar{u}y\right)  \bar{x}+\left\vert \bar{u}\right\vert
^{2}\left\vert y\right\vert ^{2}\\
-\left\vert u\right\vert ^{2}\left\vert x\right\vert ^{2}+r\left(  ux\right)
\bar{y}+ry\overline{\left(  ux\right)  }-r^{2}\left\vert y\right\vert ^{2}%
\end{array}
\right) \\
&  =\left(  r^{2}-\left\vert u\right\vert ^{2}\right)  \left(  \frac
{\left\vert x\right\vert ^{2}-\left\vert y\right\vert ^{2}}{2}\right)
+r\left(  \operatorname{Re}\left(  \bar{u}y\right)  \bar{x}+\operatorname{Re}%
\left(  ux\right)  \bar{y}\right) \\
&  =\left(  r^{2}-\left\vert u\right\vert ^{2}\right)  \left(  \frac
{\left\vert x\right\vert ^{2}-\left\vert y\right\vert ^{2}}{2}\right)
+r\left(  \left\langle \left(  \bar{u}y\right)  \bar{x},1\right\rangle
+\left\langle \left(  ux\right)  \bar{y},1\right\rangle \right) \\
&  =\left(  r^{2}-\left\vert u\right\vert ^{2}\right)  \left(  \frac
{\left\vert x\right\vert ^{2}-\left\vert y\right\vert ^{2}}{2}\right)
+2r\left\langle \bar{u},x\bar{y}\right\rangle
\end{align*}
by Lemma \ \ref{OCTBasicLemma}, Lemma\ref{conjugationLemma},
Lemma\ref{ConjugationLemma2}. And for the second one,%
\begin{align*}
&  \left(  rx+\bar{u}y\right)  \overline{\left(  ux-ry\right)  }=\left(
rx+\bar{u}y\right)  \left(  \overline{ux}-r\bar{y}\right) \\
&  =r\left\vert x\right\vert ^{2}\bar{u}+\left(  \bar{u}y\right)  \left(
\bar{x}\bar{u}\right)  -r^{2}x\bar{y}-r\bar{u}\left\vert y\right\vert ^{2}\\
&  =2r\bar{u}\left(  \frac{\left\vert x\right\vert ^{2}-\left\vert
y\right\vert ^{2}}{2}\right)  +\bar{u}\left(  \left(  y\bar{x}\right)  \bar
{u}\right)  -r^{2}x\bar{y}\\
&  =2r\bar{u}\left(  \frac{\left\vert x\right\vert ^{2}-\left\vert
y\right\vert ^{2}}{2}\right)  +\bar{u}\left(  -u\overline{\left(  y\bar
{x}\right)  }+2\left\langle \left(  y\bar{x}\right)  \bar{u},1\right\rangle
\right)  -r^{2}x\bar{y}\\
&  =2r\bar{u}\left(  \frac{\left\vert x\right\vert ^{2}-\left\vert
y\right\vert ^{2}}{2}\right)  -\left(  r^{2}+\left\vert u\right\vert
^{2}\right)  x\bar{y}+2\bar{u}\left\langle \bar{u},x\bar{y}\right\rangle \\
&  =2r\bar{u}\left(  \frac{\left\vert x\right\vert ^{2}-\left\vert
y\right\vert ^{2}}{2}\right)  -x\bar{y}+2\bar{u}\left\langle \bar{u},x\bar
{y}\right\rangle .
\end{align*}
where the third and sixth equalities are given by Moufang identities
(Proposition\ref{MoufangPro}) and $r^{2}+\left\vert u\right\vert ^{2}=1$
respectively.\ \ $\blacksquare$
\end{proof}

Again, thanks to non-associativity of octonions $\mathbb{O}$, we show the
action of $\widehat{Spin(9)}\ $is fiberwise for the Hopf map instead of
Proposition\ref{SORCH}.

\begin{proposition}
\label{OctO1SuProp}For $%
\begin{pmatrix}
x\\
y
\end{pmatrix}
,%
\begin{pmatrix}
a\\
b
\end{pmatrix}
$ $\in\mathbb{O}^{2}\smallsetminus\left\{
\begin{pmatrix}
0\\
0
\end{pmatrix}
\right\}  $\ and each $A\in\widehat{Spin(9)}$,
\[
\text{if\ \ }\Phi\left(
\begin{pmatrix}
x\\
y
\end{pmatrix}
\right)  =\Phi\left(
\begin{pmatrix}
a\\
b
\end{pmatrix}
\right)  \ \text{then is\ }\Phi\left(  A%
\begin{pmatrix}
x\\
y
\end{pmatrix}
\right)  =\Phi\left(  A%
\begin{pmatrix}
a\\
b
\end{pmatrix}
\right)  \text{.}%
\]

\end{proposition}

\begin{proof}
To prove the preposition, it is enough to show
\[
\Phi\left(  g\left(  r,u\right)
\begin{pmatrix}
x\\
y
\end{pmatrix}
\right)  =\Phi\left(  g\left(  r,u\right)
\begin{pmatrix}
a\\
b
\end{pmatrix}
\right)
\]
for each $\left(  r,u\right)  \in\mathbb{R\oplus O}\left(  1\right)  $. This
is true by Lemma \ref{HopfOctKeyLemma}$\ $since
\[
\frac{\left\vert x\right\vert ^{2}-\left\vert y\right\vert ^{2}}{2}%
=\frac{\left\vert a\right\vert ^{2}-\left\vert b\right\vert ^{2}}{2}%
,\ x\bar{y}=a\bar{b}\text{.}%
\]

\ $\blacksquare$
\end{proof}

\begin{remark}
From the point of view of Lemma\ \ref{OctonionFiberLemma}, this proposition is
equivalent to that for $%
\begin{pmatrix}
x\\
y
\end{pmatrix}
,%
\begin{pmatrix}
a\\
b
\end{pmatrix}
$ $\in\mathbb{O}^{2}\smallsetminus\left\{
\begin{pmatrix}
0\\
0
\end{pmatrix}
\right\}  $, $c\in\mathbb{O}\left(  1\right)  $, $A\in\widehat{Spin(9)}$%
\[
\text{if\ \ }\Phi_{\mathbb{O}}\left(
\begin{pmatrix}
x\\
y
\end{pmatrix}
\right)  =\Phi_{\mathbb{O}}\left(
\begin{pmatrix}
x\\
y
\end{pmatrix}
\cdot c\right)  \ \text{then is\ }\Phi_{\mathbb{O}}\left(  A%
\begin{pmatrix}
x\\
y
\end{pmatrix}
\right)  =\Phi_{\mathbb{O}}\left(  A\left(
\begin{pmatrix}
x\\
y
\end{pmatrix}
\cdot c\right)  \right)  \text{.}%
\]

\end{remark}

We observe that $\left(  g\left(  0,1\right)  \right)  ^{2}=Id,$ $g\left(
0,1\right)  g\left(  0,-1\right)  =-Id\ $in $\widehat{Spin(9)}\ $which induces
the identity map on $\mathbb{R\oplus O\ }$by $\pi\Phi$, this is all the
possible case to get the identity map. Therefore, we may consider
$\widehat{Spin(9)}/\left\{  \pm Id\right\}  $ action to $\mathbb{R\oplus O\ }%
$by $\pi\Phi\,.$%
\[%
\begin{array}
[c]{ccccccc}%
\mathbb{R\oplus O}\left(  1\right)  & \rightarrow & \widehat{Spin(9)} &
\curvearrowright & \mathbb{O}^{2} & \curvearrowleft & \mathbb{O}(1)\\
&  & \downarrow &  & \pi\Phi_{\mathbb{O}}\downarrow &  & \downarrow\\
&  & \widehat{Spin(9)}/\left\{  \pm Id\right\}  & \curvearrowright &
\mathbb{R}\oplus\mathbb{O} & \curvearrowleft & \left\{  1\right\}
\end{array}
\]
And we want to show $SU(2,\mathbb{O})/\left\{  \pm I_{2}\right\}
=\widehat{Spin(9)}/\left\{  \pm Id\right\}  $ is $SO\left(  9\right)  \ $in below.

Now by considering the Hopf map $\Phi\ $for $g\left(  r,u\right)  v$, we can
consider an induced operation of $g\left(  r,u\right)  \ $to $\pi\left(
\mathcal{H}_{2}^{0}(\mathbb{O})\right)  =\mathbb{R\oplus O}$ which turns out a
reflection of $\mathbb{R\oplus O}$ by following Proposition.

\begin{proposition}
\label{OcRefProp}For each $\left(  r,u\right)  \in\mathbb{R\oplus O}\left(
1\right)  ,$ $g\left(  r,u\right)  \ $induces a reflection
\[
2%
\begin{pmatrix}
r\\
\bar{u}%
\end{pmatrix}%
\begin{pmatrix}
r\\
\bar{u}%
\end{pmatrix}
^{t}-I_{9}%
\]
to $\pi\left(  \mathcal{H}_{2}^{0}(\mathbb{F})\right)  =\mathbb{R\oplus O}%
\ $via the Hopf map $\Phi$.
\end{proposition}

\bigskip Note that $%
\begin{pmatrix}
r\\
\bar{u}%
\end{pmatrix}
$ is considered in $\mathbb{R}^{9}$ and $%
\begin{pmatrix}
r\\
\bar{u}%
\end{pmatrix}
^{t}%
\begin{pmatrix}
s\\
w
\end{pmatrix}
=rs+\left\langle \bar{u},w\right\rangle $ for $\left(  s,w\right)
\in\mathbb{R\oplus O}$.

\begin{proof}
For $\left(  x,y\right)  \in\mathbb{O}^{2}$, we have
\[
\pi\Phi\left(  g\left(  r,u\right)
\begin{pmatrix}
x\\
y
\end{pmatrix}
\right)  =\left(  \frac{1}{2}\left(  \left\vert rx+\bar{u}y\right\vert
^{2}-\left\vert ux-ry\right\vert ^{2}\right)  ,\left(  rx+\bar{u}y\right)
\overline{\left(  ux-ry\right)  }\right)  \in\mathbb{R\oplus O}\text{.}%
\]

By Lemma \ref{HopfOctKeyLemma}, we have%
\[
\frac{1}{2}\left(  \left\vert rx+\bar{u}y\right\vert ^{2}-\left\vert
ux-ry\right\vert ^{2}\right)  =\left(  r^{2}-\left\vert u\right\vert
^{2}\right)  \left(  \frac{\left\vert x\right\vert ^{2}-\left\vert
y\right\vert ^{2}}{2}\right)  +2r\left\langle \bar{u},x\bar{y}\right\rangle
\]%
\[
\left(  rx+\bar{u}y\right)  \overline{\left(  ux-ry\right)  }=2r\bar{u}\left(
\frac{\left\vert x\right\vert ^{2}-\left\vert y\right\vert ^{2}}{2}\right)
-x\bar{y}+2\bar{u}\left\langle \bar{u},x\bar{y}\right\rangle .
\]
\ 

Since $\pi\Phi\left(
\begin{pmatrix}
x\\
y
\end{pmatrix}
\right)  =\left(  \frac{\left\vert x\right\vert ^{2}-\left\vert y\right\vert
^{2}}{2},x\bar{y}\right)  $, $g\left(  r,u\right)  $ induced a $\mathbb{R}%
$-linear map on $\mathbb{R\oplus O\ }$written as a $9\times9$ matrix%
\[
\left(
\begin{array}
[c]{cc}%
r^{2}-\left\vert u\right\vert ^{2} & 2r\bar{u}^{t}\\
2r\bar{u} & \bar{u}\bar{u}^{t}%
\end{array}
\right)
\]
where $\bar{u}\ $is considered as a column in $\mathbb{R}^{8}$. Since
$r^{2}+\left\vert u\right\vert ^{2}=1$, this can be rewritten as%
\[
2\left(
\begin{array}
[c]{cc}%
r^{2} & r\bar{u}^{t}\\
r\bar{u} & \bar{u}\bar{u}^{t}%
\end{array}
\right)  -I_{9}\text{.}%
\]

\ $\blacksquare$
\end{proof}

By Proposition\ref{OcRefProp}, we define a map%
\[%
\begin{array}
[c]{cccc}%
\rho_{\mathbb{O}}: & \mathbb{R\oplus O}\left(  1\right)  & \rightarrow &
End_{_{\mathbb{R}}}\left(  \mathbb{R\oplus O}\right) \\
& \left(  r,u\right)  & \longmapsto & 2%
\begin{pmatrix}
r\\
\bar{u}%
\end{pmatrix}%
\begin{pmatrix}
r\\
\bar{u}%
\end{pmatrix}
^{t}-I_{9}.
\end{array}
\]
In the following Lemma, we show $-\rho_{\mathbb{O}}\left(  g\left(
r,u\right)  \right)  \ $is a reflection on $\mathbb{R\oplus O}$.

\begin{lemma}
\label{OctRecVerLemma} For each $\left(  r,u\right)  \in\mathbb{R\oplus
O}\left(  1\right)  ,-\rho_{\mathbb{O}}\left(  \left(  r,u\right)  \right)
\ $is a reflection of $\pi\left(  \mathcal{H}_{2}^{0}(\mathbb{F})\right)
=\mathbb{R\oplus O}\ $to the hyperplane perpendicular to $\left(  r,\bar
{u}\right)  \ $in $\pi\left(  \mathcal{H}_{2}^{0}(\mathbb{F})\right)  $.
\end{lemma}

\begin{proof}
If $\left(  s,w\right)  \in\mathbb{R\oplus O\ }$perpendicular to $\left(
r,\bar{u}\right)  \ $as vectors in $\mathbb{R}^{9}$, namely,%
\[%
\begin{pmatrix}
r\\
\bar{u}%
\end{pmatrix}
^{t}%
\begin{pmatrix}
s\\
w
\end{pmatrix}
=0\text{,}%
\]
then
\[
-\rho\left(  \left(  r,u\right)  \right)
\begin{pmatrix}
s\\
w
\end{pmatrix}
=\left(  -2%
\begin{pmatrix}
r\\
\bar{u}%
\end{pmatrix}%
\begin{pmatrix}
r\\
\bar{u}%
\end{pmatrix}
^{t}+I_{9}\right)
\begin{pmatrix}
s\\
w
\end{pmatrix}
=%
\begin{pmatrix}
s\\
w
\end{pmatrix}
\text{.}%
\]
And if $\left(  s,w\right)  \in\mathbb{R\oplus O\ }$parallel to $\left(
r,\bar{u}\right)  \ $as vectors in $\mathbb{R}^{9}$, namely,%
\[%
\begin{pmatrix}
s\\
w
\end{pmatrix}
=c%
\begin{pmatrix}
r\\
\bar{u}%
\end{pmatrix}
\text{\ for }c\in\mathbb{R}%
\]
then
\begin{align*}
-\rho\left(  \left(  r,u\right)  \right)
\begin{pmatrix}
s\\
w
\end{pmatrix}
&  =\left(  -2%
\begin{pmatrix}
r\\
\bar{u}%
\end{pmatrix}%
\begin{pmatrix}
r\\
\bar{u}%
\end{pmatrix}
^{t}+I_{9}\right)
\begin{pmatrix}
cr\\
c\bar{u}%
\end{pmatrix}
\\
&  =-c%
\begin{pmatrix}
r\\
\bar{u}%
\end{pmatrix}
=-%
\begin{pmatrix}
s\\
w
\end{pmatrix}
\text{.}%
\end{align*}
Thus $-\rho\left(  \left(  r,u\right)  \right)  \ $is a reflection of
$\pi\left(  \mathcal{H}_{2}^{0}(\mathbb{F})\right)  =\mathbb{R\oplus O}$ of
the hyperplane perpendicular to $\left(  r,\bar{u}\right)  $. \ $\blacksquare$
\end{proof}

\begin{corollary}
\label{OctRecCoro}
\[
\left\{  \left.  -\rho_{\mathbb{O}}\left(  r,u\right)  \right\vert \ \left(
r,u\right)  \in\mathbb{R\oplus O}\left(  1\right)  \text{ }\right\}  \subset
End_{_{\mathbb{R}}}\left(  \mathbb{R\oplus O}\right)
\]
generates $O\left(  \mathbb{R\oplus O}\right)  =O\left(  9\right)  $.
\end{corollary}

\begin{proof}
Each reflection of $\mathbb{R\oplus O}$ can be given as $-\rho_{\mathbb{O}%
}\left(  r,u\right)  $ for a $\left(  r,u\right)  \in\mathbb{R\oplus O}$ with
$r^{2}+\left\vert u\right\vert ^{2}=1$. By Cartan-Dieudonn\'{e} Theorem,
$O\left(  \mathbb{R\oplus O}\right)  $ is generated by%
\[
\left\{  \left.  -\rho_{\mathbb{O}}\left(  r,u\right)  \right\vert \ \left(
r,u\right)  \in\mathbb{R\oplus O}\left(  1\right)  \text{ }\right\}  \text{.}%
\]
\ $\blacksquare$ $\mathbb{\ }$
\end{proof}

Since $-\rho_{\mathbb{O}}\left(  r,u\right)  \ $is a reflection of
$\mathbb{R\oplus O}$,$\ $
\[
\det\rho_{\mathbb{O}}\left(  r,u\right)  =\det\left(  -\left(  -\rho
_{\mathbb{O}}\left(  r,u\right)  \right)  \right)  =\left(  -1\right)
^{9}\left(  -1\right)  =1
\]
for a $9\times9$ matrix $\rho_{\mathbb{O}}\left(  r,u\right)  $. Thus
$\rho_{\mathbb{O}}\left(  r,u\right)  \ $is an element of $SO\left(  9\right)
$. Moreover because$\ \widehat{Spin(9)}\ $is generated by
\[
\left\{  \left.  g\left(  r,u\right)  \right\vert \ r\in\mathbb{R}%
\text{,}\ u\in\mathbb{O}\text{, and }r^{2}+\left\vert u\right\vert
^{2}=1\text{ }\right\}
\]
as the Remark of Lemma \ref{Spin9Lemma}\thinspace, the map $\rho_{\mathbb{O}%
}\ $can be extended to a group homomorphism
\[%
\begin{array}
[c]{cccc}%
\Gamma_{\mathbb{O}}: & \ \widehat{Spin(9)} & \rightarrow & SO\left(
\mathbb{R\oplus O}\right)  \subset End_{_{\mathbb{R}}}\left(  \mathbb{R\oplus
O}\right)  .
\end{array}
\]
By Corollary \ref{OctRecCoro}, $\Gamma_{\mathbb{O}}\ $is onto and
$\Gamma_{\mathbb{O}}^{-1}\left(  I_{9}\right)  =\left\{  \pm Id\right\}  $.
Therefore $\Gamma_{\mathbb{O}}\ $is a $2$-covering to $SO\left(
\mathbb{R\oplus O}\right)  $.

In summary we obtain%
\[
\underset{\text{\textbf{Diagram AO}}}{%
\begin{array}
[c]{ccccccc}%
\mathbb{R\oplus O}\left(  1\right)  & \underrightarrow{g} & \widehat
{Spin(9)}\simeq SU(2,\mathbb{O}) & \curvearrowright & \mathbb{O}^{2} &
\curvearrowleft & \mathbb{O}(1)\\
&  &  &  &  &  & \\
&  & \Gamma_{\mathbb{O}}\downarrow &  & \pi\Phi\downarrow &  & \downarrow\\
&  &  &  &  &  & \\
&  & \widehat{Spin(9)}/\left\{  \pm Id\right\}  \simeq SO\left(
\mathbb{R\oplus O}\right)  & \curvearrowright & \mathbb{R}\oplus\mathbb{O} &
\curvearrowleft & \left\{  1\right\}
\end{array}
}%
\]

\section{Stiefel manifolds and Polygon spaces for $\mathbb{F}=\mathbb{R}%
,\mathbb{C},\mathbb{H}$}

Homeomorphisms of polygon spaces $\mathcal{P}_{k}(\mathbb{R}^{3})$ and complex
2-grassamannians gives a bridge to explore the geometry of polygon spaces by
borrowing the geometric structure of the grassmannians (\cite{HK}). Hausmann
and Knutson showed that
\[
\mathcal{P}_{k}(\mathbb{R}^{3})\simeq U(1)^{k}\backslash Gr_{\mathbb{C}}(2,k)
\]
using the Hopf map $\mathbb{H}\rightarrow Im(\mathbb{H})$ by $q\mapsto
\overline{q}iq$. In addition, a restriction of the Hopf map on $\mathbb{C}%
\subset\mathbb{H}$ gives homeomorphisms
\[
\mathcal{P}_{k}(\mathbb{R}^{2})\simeq O(1)^{k}\backslash Gr_{\mathbb{R}}(2,k)
\]
and
\[
\widetilde{\mathcal{P}}_{k}(\mathbb{R}^{2})\simeq O(1)^{k}\backslash
\widetilde{Gr}_{\mathbb{R}}(2,k)
\]
where $\widetilde{Gr}_{\mathbb{R}}(2,k)$ is the oriented grassmannian.
Moreover the authors of \cite{HK}\ expected the octonionic version of their
construction of polygon spaces.

In this section, we introduce another approach along the spin representation
introduced in Table 1. From approach, we introduce another Hopf map so that
polygon spaces for $\mathbb{R}^{2}$\ and $\mathbb{R}^{3}$ in \cite{HK}$\ $can
be considered as ones for $\mathbb{R}\oplus\mathbb{R}$ and $\mathbb{R}%
\oplus\mathbb{C}$ induced by $Spin\left(  2\right)  \simeq SO(2)\ $and
$Spin\left(  3\right)  \simeq SU(2)$ representation along our Hopf maps.
Therefrom, we also provide a way connecting polygon space over $\mathbb{R}%
^{2},\mathbb{R}^{3}$ and $2$-grassmannians by the spin representation of
$Spin\left(  2\right)  \ $and $Spin\left(  3\right)  $. From this point view,
the construction of polygon space for $\mathbb{H}$ in \cite{HK}$\ $is rather
for $\mathbb{C}^{2}$, and it is natural to consider the extension to
$\mathbb{H}^{2}$ and $\mathbb{O}^{2}$ via $Spin\left(  5\right)  \simeq
Sp(2)\ $and $Spin\left(  9\right)  \simeq SU(2,\mathbb{O}\dot{)}$. The
connection of polygon spaces over $\mathbb{R}^{2},\mathbb{R}^{3}$ and
$2$-grassmannians is shown from the spin representation, and the related
extensions along the normed division algebras are presented.

\subsection{Polygon spaces and $\mathbb{R},\mathbb{C}$, $\mathbb{H}$}

From our approach, we show how the polygon spaces $\widetilde{\mathcal{P}}%
_{k}(\mathbb{R}^{n})\ $for $n=2,3,5\ $are connected to $2$-Grassmannians.

In \cite{HK}, Hausmann and Knutson studied the homeomorphism between the
$3$-dimensional polygon space and the complex $2$-Grassmannian, and posed an
open question regarding the extension of their approach. Specifically, they
asked whether their construction, based on the Hopf map $\mathbb{H}\rightarrow
Im(\mathbb{H})$ defined on the quaternions, could be generalized to an
octonionic setting. The direction of extending Theorem \ref{thm1} using the
spin representation provides a possible answer to this open question. In the
case $\mathbb{F}=\mathbb{H},$ it is important to note here that, unlike
$\mathbb{R}$ and $\mathbb{C}$, the quaternions $\mathbb{H}$ are
non-commutative, and this must be carefully taken into account throughout the construction.

Define a surjection $\Phi^{k}$ by extending the composition of the Hopf map
$\Phi$ and isomorphism $\pi$ on $2$-Stiefel manifold: let $\mathbb{F}%
=\mathbb{R},\mathbb{C}$, $\mathbb{H}$ and $V=\mathbb{R}^{2},\mathbb{R}^{3}$,
$\mathbb{R}^{5}$ respectively,%
\[%
\begin{array}
[c]{cccc}%
\Phi^{k}: & V_{\mathbb{F}}(2,k) & \rightarrow & \mathcal{M}_{k}(\mathbb{R}%
^{1+dim_{\mathbb{R}}\mathbb{F}})\\
&
\begin{pmatrix}
x_{1} & \cdots & x_{k}\\
y_{1} & \cdots & y_{k}%
\end{pmatrix}
& \mapsto & \left(  \pi\circ\Phi\left(
\begin{pmatrix}
x_{1}\\
y_{1}%
\end{pmatrix}
\right)  \cdots\pi\circ\Phi\left(
\begin{pmatrix}
x_{k}\\
y_{k}%
\end{pmatrix}
\right)  \right)
\end{array}
\]
Here, for any $i$, the composition $\pi\circ\Phi\left(
\begin{pmatrix}
x_{i}\\
y_{i}%
\end{pmatrix}
\right)  $ maps to $\mathbb{R}\oplus\mathbb{F}$. By naturally identifying
$\mathbb{R}\oplus\mathbb{F}$ with $\mathbb{R}^{1+dim_{\mathbb{R}}\mathbb{F}}%
$\ and, for convenience, we will use the same notation $\pi\circ\Phi$ for
$\mathbb{F}^{2}\rightarrow\mathbb{R}\oplus\mathbb{F}\rightarrow\mathbb{R}%
^{1+dim_{\mathbb{R}}\mathbb{F}}$ to define $\Phi^{k}$.

\begin{lemma}
\label{Phik} Let $\mathbb{F}=\mathbb{R},\mathbb{C}$, $\mathbb{H}$ and for each
case, let the corresponding $V$ be $\mathbb{R}^{1+dim_{\mathbb{R}}\mathbb{F}}%
$. Then:

(1) $\Phi^{k}$ is surjective.

(2) For each $[P]\in\mathcal{M}_{k}(V)$, $(\Phi^{k})^{-1}([P])\simeq
\mathbb{F}(1)^{l}$ for some $2\leq l\leq k$. In the case of non-degenerate
$[P]$, we have $l=k$.
\end{lemma}

\begin{proof}
(1) The orthonormality condition of $V_{\mathbb{F}}(2,k)$ corresponds, on the
polygon space side, to the condition that the polygons are closed and have
fixed perimeter $1$. For $A\in V_{\mathbb{F}}(2,k)$, because
\[
I_{2}=AA^{\ast}=%
\begin{pmatrix}
x_{1} & \cdots & x_{k}\\
y_{1} & \cdots & y_{k}%
\end{pmatrix}%
\begin{pmatrix}
x_{1} & \cdots & x_{k}\\
y_{1} & \cdots & y_{k}%
\end{pmatrix}
^{\ast}%
\]
we have
\[
\sum_{i=1}^{k}x_{i}^{2}=\sum_{i=1}^{k}y_{i}^{2}=1,\sum_{i=1}^{k}x_{i}\bar
{y}_{i}=0\text{.}%
\]
Therefore, the polygons%
\[
\Phi^{k}\left(
\begin{pmatrix}
x_{1} & \cdots & x_{k}\\
y_{1} & \cdots & y_{k}%
\end{pmatrix}
\right)  =\left(  \frac{\left\vert x_{1}\right\vert ^{2}-\left\vert
y_{1}\right\vert ^{2}}{2}+x_{1}\bar{y}_{1},\cdots,\frac{\left\vert
x_{k}\right\vert ^{2}-\left\vert y_{k}\right\vert ^{2}}{2}+x_{k}\bar{y}%
_{k}\right)  \in\left(  \mathbb{R}\oplus\mathbb{F}\right)  ^{k}\text{,}%
\]
are closed since
\[
\sum_{i=1}^{k}\left(  \frac{\left\vert x_{i}\right\vert ^{2}-\left\vert
y_{i}\right\vert ^{2}}{2}+x_{i}\bar{y}_{i}\right)  =\frac{1}{2}\left(
\sum_{i=1}^{k}\left\vert x_{i}\right\vert ^{2}-\sum_{i=1}^{k}\left\vert
y_{i}\right\vert ^{2}\right)  +\sum_{i=1}^{k}x_{i}\bar{y}_{i}=0
\]
and have fixed perimeter $1\ $because%
\begin{align*}
\sum_{i=1}^{k}\left\vert \frac{\left\vert x_{i}\right\vert ^{2}-\left\vert
y_{i}\right\vert ^{2}}{2}+x_{i}\bar{y}_{i}\right\vert  &  =\sum_{i=1}^{k}%
\sqrt{\left(  \frac{\left\vert x_{i}\right\vert ^{2}-\left\vert y_{i}%
\right\vert ^{2}}{2}\right)  ^{2}+\left\vert x_{i}\bar{y}_{i}\right\vert ^{2}%
}\\
&  =\sum_{i=1}^{k}\frac{\left\vert x_{i}\right\vert ^{2}+\left\vert
y_{i}\right\vert ^{2}}{2}=1\text{.}%
\end{align*}
Therefore, it is clear that the image of $\Phi^{k}$ is contained in
$\mathcal{M}_{k}(V)$.

Conversely, to show that $\Phi^{k}$ is surjective, we consider the way to
recover an orthonormal $2$-frame from a given polygon in $\mathcal{M}_{k}(V)$.
Each edge of the polygon corresponds to a vector $\vec{v_{i}}$, and finding a
preimage of the polygon amounts to finding a pair $%
\begin{pmatrix}
x_{i}\\
y_{i}%
\end{pmatrix}
\in\mathbb{F}^{2}$ such that $(\pi\circ\Phi)\left(
\begin{pmatrix}
x_{i}\\
y_{i}%
\end{pmatrix}
\right)  =\vec{v_{i}}$. This is exactly the inverse image under $\pi\circ\Phi$
for each $i=1,\cdots,k$, which exists by Lemma \ref{surjection}. Therefore,
$\Phi^{k}$ is surjective onto $\mathcal{M}_{k}(V)$.

The structure of the fiber of $\Phi^{k}$ also follows directly from Lemma
\ref{surjection}. Recall that for $(\lambda,\alpha)\neq(0,0)$, we have
$\Phi^{-1}\left(
\begin{pmatrix}
\lambda & \alpha\\
\overline{\alpha} & -\lambda
\end{pmatrix}
\right)  \simeq\mathbb{F}(1)$, while for $(\lambda,\alpha)=(0,0)$,$\ \Phi
^{-1}\left(
\begin{pmatrix}
0 & 0\\
0 & 0
\end{pmatrix}
\right)  \simeq\left\{
\begin{pmatrix}
0\\
0
\end{pmatrix}
\right\}  $.

However, by the definition of $\mathcal{M}_{k}(V)$, zero polygons (i.e.,
polygons with all edges zero) are excluded. Therefore, none of the entries in
a polygon configuration can correspond to the zero element in $\mathbb{R}%
\oplus\mathbb{F\ }$simultaneously. As a result, for a degenerate polygon
$[P]$, the fiber of $\Phi^{k}$ is isomorphic to $\mathbb{F}(1)^{l}$ for some
$l<k$, while for a non-degenerate polygon, the fiber is isomorphic to
$\mathbb{F}(1)^{k}$.\ \ \ \ \ $\blacksquare$
\end{proof}

As both the $\mathbb{F}(1)$ and the special unitary group $SU(2,\mathbb{F})$
act on $\mathbb{F}^{2}$, these actions naturally extend to $V_{\mathbb{F}%
}(2,k)$. Since the fiber $\mathbb{F}(1)^{k}$ of $\Phi^{k}$ also has a natural
group structure, and then we can define a right action of $\mathbb{F}(1)^{k}$
on the Stiefel manifold $V_{\mathbb{F}}(2,k)$. For $c=(c_{1},\cdots,c_{k}%
)\in\mathbb{F}(1)^{k}$ and $X=%
\begin{pmatrix}
x_{1} & \cdots & x_{k}\\
y_{1} & \cdots & y_{k}%
\end{pmatrix}
\in V_{\mathbb{F}}(2,k)$, the right action is defined column-wise by\quad\
\[
X\cdot c:=\left(
\begin{pmatrix}
x_{1}\\
y_{1}%
\end{pmatrix}
\cdot c_{1}\text{,}\cdots,%
\begin{pmatrix}
x_{k}\\
y_{k}%
\end{pmatrix}
\cdot c_{k}\right)  .
\]
This defines an equivalence relation on $V_{\mathbb{F}}(2,k)$, and yields the
following diagram with a well-defined map $\widetilde{\Phi^{k}}$ satisfying
$\Phi^{k}=\widetilde{\Phi^{k}}\circ\pi$.%

\[%
\begin{matrix}
\Phi^{k}: & V_{\mathbb{F}}(2,k) & \longrightarrow & \mathcal{M}_{k}(V)\\
&  &  & \\
& \pi\ \big\downarrow & \nearrow & \widetilde{\Phi^{k}}\\
&  &  & \\
& V_{\mathbb{F}}(2,k)/\mathbb{F}(1)^{k} &  &
\end{matrix}
\]

The special unitary groups $SU(2,\mathbb{F})$ also acts on $V_{\mathbb{F}%
}(2,k)$ by applying its standard action on $\mathbb{F}^{2}$ to each column.
That is, for $A\in SU(2,\mathbb{F})$ and $X=%
\begin{pmatrix}
x_{1} & \cdots & x_{k}\\
y_{1} & \cdots & y_{k}%
\end{pmatrix}
\in V_{\mathbb{F}}(2,k)$, we define
\[
A\cdot X:=\left(  A%
\begin{pmatrix}
x_{1}\\
y_{1}%
\end{pmatrix}
\cdots A%
\begin{pmatrix}
x_{k}\\
y_{k}%
\end{pmatrix}
\right)  =AX.
\]
Using the associativity of $\mathbb{F}$, this action is well-defined since
\[
(A\cdot X)(A\cdot X)^{\ast}=(AX)(AX)^{\ast}=(AX)(X^{\ast}A^{\ast})=A(XX^{\ast
})A^{\ast}=I_{2}.
\]

\begin{proposition}
\label{SO} For $\mathbb{F}=\mathbb{R},\mathbb{C}$, $\mathbb{H}$, the standard
action of $SU(2,\mathbb{F})$ on the Stiefel manifold $V_{\mathbb{F}}(2,k)$
induces an action of $SO(1+\dim_{\mathbb{R}}\mathbb{F})$ on the polygon space
$\mathcal{M}_{k}(V)$, where $V$ is the $(1+\dim_{\mathbb{R}}\mathbb{F}%
)$-dimensional Euclidean space.
\end{proposition}

\begin{proof}
Since $SU(2,\mathbb{F})$ acts column-wise on $V_{\mathbb{F}}(2,k)$, apply the
Proposition \ref{SORCH} to each column. Then there is induced action of
$SO(1+dim_{\mathbb{R}}\mathbb{F})$ on $\mathcal{M}_{k}(V)$, for each edge
vector $\pi\circ\Phi\left(
\begin{pmatrix}
x_{i}\\
y_{i}%
\end{pmatrix}
\right)  \in V$, the group $SO(1+\dim_{\mathbb{R}}\mathbb{F})$ acts on the
vector in $V$.\ \ $\blacksquare$
\end{proof}

The subgroups\ $\{\pm\left(  1,1,\cdots,1\right)  \}$ of $\mathbb{F}(1)^{k}$
and $\{\pm I\}\ $of$\ SU(2,\mathbb{F})$ act identically on the Stiefel
manifold $V_{\mathbb{F}}(2,k)$. Recall that the elements $c\in\mathbb{F}(1)$
that yield the same result under both the right $\mathbb{F}(1)$ action and the
standard action of $SU(2,\mathbb{F})$ on $\mathbb{F}^{2}$ must satisfy $c\in
Z(\mathbb{F})\cap\mathbb{F}(1)$, which turns out to be $\{\pm1\}$. This under
both the right $V_{\mathbb{F}}(2,k)$, where each column corresponds to an
element of $\mathbb{F}^{2}$. Hence, the elements of $\mathbb{F}(1)^{k}$ and
$SU(2,\mathbb{F})$ that act identically on the entire Stiefel manifold are
precisely $\{\pm\left(  1,1,\cdots,1\right)  \}\ $and $\{\pm I\}$.

\begin{theorem}
\label{thm1} For $\mathbb{F}=\mathbb{R},\mathbb{C}$, $\mathbb{H}$, there are
homeomorphisms of polygon spaces:%
\begin{align*}
\widetilde{\mathcal{P}}_{k}(\mathbb{R}^{2})  &  \simeq\widetilde
{Gr}_{\mathbb{R}}(2,k)/O(1)^{k}\\
\widetilde{\mathcal{P}}_{k}(\mathbb{R}^{3})  &  \simeq SU(2)\backslash
V_{\mathbb{C}}(2,k)/U(1)^{k}\\
\widetilde{\mathcal{P}}_{k}(\mathbb{R}^{5})  &  \simeq Gr_{\mathbb{H}%
}(2,k)/\mathbb{H}(1)^{k}%
\end{align*}

\end{theorem}

\begin{proof}
From the above, $\widetilde{\Phi^{k}}$ is a homeomorphism between the quotient
space $V_{\mathbb{F}}(2,k)/\mathbb{F}(1)^{k}$ and the space $\mathcal{M}%
_{k}(V)$ of polygons with fixed side lengths. Furthermore, for the special
unitary group $SU(2,\mathbb{F})$, we consider its group action on
$V_{\mathbb{F}}(2,k)$, which gives rise to another quotient space. Together,
these yield the left column of the diagram below.

On the other side, by Proposition \ref{SO}, there is the induced
$SO(1+dim\mathbb{F})$ action on $\mathcal{M}_{k}(V)$. Taking the quotient by
this action gives the polygon space $\widetilde{\mathcal{P}}_{k}%
(\mathbb{R}^{1+dim\mathbb{F}})$. Therefore, we get a homeomorphism of the
quotient space of the Stiefel manifold and the polygon space. This leads to
the commutative diagram below.%
\[%
\begin{matrix}
V_{\mathbb{F}}(2,k)/\mathbb{F}(1)^{k} & \xrightarrow{ \widetilde{\Phi}^k} &
\mathcal{M}_{k}(\mathbb{R}^{1+dim\mathbb{F}})\\
&  & \\
SU(2,\mathbb{F})\ \big\downarrow &  & \big\downarrow\ SO(1+dim\mathbb{F})\\
&  & \\
SU(2,\mathbb{F})\backslash V_{\mathbb{F}}(2,k)/\mathbb{F}(1)^{k} &
\xrightarrow{\simeq} & \widetilde{\mathcal{P}}_{k}(\mathbb{R}^{1+dim\mathbb{F}%
})
\end{matrix}
\]

\ $\blacksquare$
\end{proof}

\subsection{Polygons spaces and$\ \mathbb{O}$}

We consider octonionic extension of Theorem \ref{thm1}. Since octonions
$\mathbb{O}$ is not associative the constructions for $\mathbb{F}%
=\mathbb{R},\mathbb{C}$, $\mathbb{H\ }$may not be extend to $\mathbb{O\ }%
$directly. Moreover $\mathbb{O}\left(  1\right)  \ $is not a group but a
Moufang loop, and $V_{\mathbb{O}}(k,n)\ $is not considered as the spaces of
$k$-frame in $\mathbb{O}^{n}$.

As in $\mathbb{F}=\mathbb{R},\mathbb{C}$, $\mathbb{H\ }$cases, we identify the
direct sum $\mathbb{R}\oplus\mathbb{O}$ with $\mathbb{R}^{1+8}$, and we define
the map%
\[%
\begin{array}
[c]{cccc}%
\Phi_{\mathbb{O}}^{k}: & V_{\mathbb{O}}(2,k) & \rightarrow & \mathcal{M}%
_{k}(\mathbb{R}^{9})\\
&
\begin{pmatrix}
x_{1} & \cdots & x_{k}\\
y_{1} & \cdots & y_{k}%
\end{pmatrix}
& \mapsto & \left(  \pi\circ\Phi\left(
\begin{pmatrix}
x_{1}\\
y_{1}%
\end{pmatrix}
\right)  \cdots\pi\circ\Phi\left(
\begin{pmatrix}
x_{k}\\
y_{k}%
\end{pmatrix}
\right)  \right)
\end{array}
\]
by using Hopf map $\Phi_{\mathbb{O}}\ $for $\mathbb{O}^{2}$.

By following the argument of Lemma \ref{Phik}, we also obtain the following Lemma.

\begin{lemma}
For $\mathbb{O}$ and the corresponding $V$ be $\mathbb{R}^{1+8}$,

(1) $\Phi_{\mathbb{O}}^{k}$ is surjective.

(2) For each $[P]\in\mathcal{M}_{k}(V)$, $(\Phi_{\mathbb{O}}^{k}%
)^{-1}([P])\simeq\mathbb{O}(1)^{l}$ for some $2\leq l\leq k$. In the case of
non-degenerate $[P]$, we have $l=k$.
\end{lemma}

The right multiplication of Moufang Loop $\mathbb{O}(1)$ and the left action
of the special unitary group $SU(2,\mathbb{O})$ on $\mathbb{O}^{2}$ naturally
extend to $V_{\mathbb{O}}(2,k)$. For $c=(c_{1},\cdots,c_{k})\in\mathbb{O}%
(1)^{k}$ and $X=%
\begin{pmatrix}
x_{1} & \cdots & x_{k}\\
y_{1} & \cdots & y_{k}%
\end{pmatrix}
\in V_{\mathbb{O}}(2,k)$, the right action is defined column-wise by\quad\
\[
X\cdot c:=\left(
\begin{pmatrix}
x_{1}\\
y_{1}%
\end{pmatrix}
\cdot c_{1}\text{,}\cdots,%
\begin{pmatrix}
x_{k}\\
y_{k}%
\end{pmatrix}
\cdot c_{k}\right)  .
\]
By applying Lemma \ref{OctonionFiberLemma}, $\mathbb{O}(1)^{k}$-multiplication
yields the following diagram with a well-defined map $\widetilde
{\Phi_{\mathbb{O}}^{k}}$ satisfying $\Phi_{\mathbb{O}}^{k}=\widetilde
{\Phi_{\mathbb{O}}^{k}}\circ\pi$.%

\[%
\begin{matrix}
\Phi_{\mathbb{O}}^{k}: & V_{\mathbb{O}}(2,k) & \longrightarrow &
\mathcal{M}_{k}(V)\\
&  &  & \\
& \pi\big\downarrow & \nearrow & \widetilde{\Phi_{\mathbb{O}}^{k}}\\
&  &  & \\
& V_{\mathbb{O}}(2,k)/\mathbb{O}(1)^{k} &  &
\end{matrix}
,
\]

The special unitary groups $SU(2,\mathbb{O})\simeq\widehat{Spin(9)}$ also acts
on $V_{\mathbb{O}}(2,k)$ by applying its spin action on $\mathbb{O}^{2}$ to
each column. we define
\[
A\cdot X:=\left(  A%
\begin{pmatrix}
x_{1}\\
y_{1}%
\end{pmatrix}
\cdots A%
\begin{pmatrix}
x_{k}\\
y_{k}%
\end{pmatrix}
\right)  =AX
\]
for $A\in SU(2,\mathbb{O})$. By Proposition \ref{OctO1SuProp}, the
$SU(2,\mathbb{O})$-action to $V_{\mathbb{O}}(2,k)\ $induces $SU(2,\mathbb{O}%
)$-action\ to $V_{\mathbb{O}}(2,k)/\mathbb{O}(1)^{k}$. By Corollary
\ref{OctRecCoro} and diagram AO in subsection \ref{subsectionOandGroupaction}%
\thinspace, this action corresponding to the $SO(1+8)\ $action on
$\mathcal{M}_{k}(\mathbb{R}^{1+8})$. Therefore the octonionic version of
Proposition \ref{SO}$~$is established, and we have the commuting diagram
\[%
\begin{matrix}
V_{\mathbb{O}}(2,k)/\mathbb{O}(1)^{k} & \xrightarrow{ \widetilde{\Phi}^k} &
\mathcal{M}_{k}(\mathbb{R}^{1+8})\\
&  & \\
SU(2,\mathbb{O})\ \big\downarrow &  & \big\downarrow\ SO(1+8)\\
&  & \\
SU(2,\mathbb{O})\backslash V_{\mathbb{O}}(2,k)/\mathbb{O}(1)^{k} &
\xrightarrow{\simeq} & \widetilde{\mathcal{P}}_{k}(\mathbb{R}^{1+8})
\end{matrix}
.
\]
Therefrom, we get the following theorem for octonions.

\begin{theorem}
\label{thm2} The $Spin(9)$-action of $SU(2,\mathbb{O})$ induces a
homeomorphisms
\[
\widetilde{\mathcal{P}}_{k}(\mathbb{R}^{9})\simeq SU(2,\mathbb{O})\backslash
V_{\mathbb{O}}(2,k)/\mathbb{O}(1)^{k}%
\]
for polygon spaces $\widetilde{\mathcal{P}}_{k}(\mathbb{R}^{9})$.
\end{theorem}

\begin{remark}
Our study has primarily relied on the spin representation theory in low
dimensions. A natural direction for extending this work is to consider
replacing the field $\mathbb{F\ }$with the octonions $\mathbb{O}$, the final
member in the family of normed division algebras. However, unlike the real
numbers, complex numbers, and quaternions, the octonions are not associative.
As a result, standard constructions that work for $\mathbb{F}=\mathbb{R}%
,\mathbb{C},\mathbb{H}$ often fail or require significant modification in the
octonionic setting. Despite this, we have attempted to extend our framework to
the case $\mathbb{F}=\mathbb{O}$, and found that, through adjustments of the
earlier constructions, a similar theory can indeed be formulated considering
$SU(2,\mathbb{O})\ $as $Spin\left(  9\right)  \ $and related spinors. We
remark the extension to $\mathbb{O\ }$can also be considered for $Spin\left(
8\right)  $ to $\mathbb{O}^{2}$ and the triality of $Spin\left(  8\right)
\ $induces various interesting structures on polygon spaces. We discuss this
in another article.
\end{remark}

\textbf{Acknowledgements: }This research was supported by the National
Research Foundation of Korea(NRF) grant funded by the Korea
government(MSIT)(No. RS-2024-00359647). This research was also supported by
Basic Science Research Program (Priority Research Institute) through the
National Research Foundation of Korea (NRF) funded by the Ministry of
Education (2021R1A6A1A10039823).

\bigskip

\bigskip

Addresses

Eunjeong Lee (ddwjd92@gmail.com)

Department of Mathematics, Ewha Womans University, Seoul 120-750, Korea

Jae-Hyouk Lee (jaehyoukl@ewha.ac.kr)

Department of Mathematics, Ewha Womans University, Seoul 120-750, Korea

\end{document}